\documentclass{article}
\usepackage[shortlabels]{enumitem}
\usepackage{amsmath}
\usepackage{amssymb}
\usepackage{amsthm}
\usepackage{abstract}
\usepackage{xcolor}
\usepackage{comment}
\usepackage{mathtools}
\usepackage{graphicx}
\usepackage{caption}
\usepackage{subcaption}
\usepackage{float}
\usepackage{hyperref}
\usepackage{cleveref}
\usepackage{tikz}
\usetikzlibrary{positioning, shapes.geometric, arrows.meta}
\hypersetup{
    colorlinks=true,
    linkcolor=blue,
    filecolor=magenta,      
    urlcolor=cyan
    }
\makeatletter
\newcommand*{\rom}[1]{\expandafter\@slowromancap\romannumeral #1@}
\newcommand*{\barfix}[2][.175ex]{%
  \mathpalette{\@barfix{#1}}{#2}%
}
\newcommand*{\@barfix}[3]{%
  \vbox{%
    \kern#1\relax
    \hbox{$#2#3\m@th$}%
  }%
}
\makeatother

\newtheorem{theorem}{Theorem}
\newtheorem{thm}{Theorem}[section]

\newtheorem{lemma}[thm]{Lemma}
\newtheorem{proposition}[thm]{Proposition}
\newtheorem{claim}[thm]{Claim}
\newtheorem{conjecture}[thm]{Conjecture}

\theoremstyle{definition}
\newtheorem{remark}[thm]{Remark}

\usepackage[top=21 mm, bottom=21 mm, left=20 mm, right= 20 mm]{geometry} 
\newcommand{\eps}{\varepsilon}

\newcommand{\e}{\mathrm{e}}

\newcommand{\cA}{\mathcal{A}}
\newcommand{\cB}{\mathcal{B}}
\newcommand{\cC}{\mathcal{C}}
\newcommand{\cD}{\mathcal{D}}

\newcommand{\cP}{\mathcal{P}}
\newcommand{\cQ}{\mathcal{Q}}

\title{The Mihail--Vazirani conjecture and strong edge-expansion in random $0/1$ polytopes} 

\author{%
Micha Christoph\thanks{D-MATH ETH Zurich, R\"amistrasse 101, 8092 Z\"urich, Switzerland}
\and Sahar Diskin\footnotemark[1]
\and Lyuben Lichev\thanks{Institute of Statistics and Mathematical Methods in Economics, TU Wien, A-1040 Vienna, Austria}
\and Benny Sudakov\footnotemark[1]
}

\date{}

\begin{document}
\maketitle

\begin{abstract}
We study the edge-expansion of the graph of a random $0/1$ polytope
$P^d_p$, defined as the convex hull of a random subset of the points in $\{0,1\}^d$ where every point is retained independently and with probability $p$.
This problem was introduced more than twenty years ago in a work of Gillmann and Kaibel, 
and has been extensively studied ever since.

We prove that, for every fixed $\eps>0$ and every $p\in(0,1-\eps]$, with high
probability the graph of $P^d_p$ has edge-expansion $\Theta(d)$. This improves the
previously best known bound due to Ferber, Krivelevich, Sales and Samotij, and
verifies, in a strong form, the celebrated Mihail--Vazirani conjecture for random
$0/1$ polytopes. Although the expansion factor $\Theta(d)$ is typically best
possible for $p\ge 1/2+\eps$, we also show that the behaviour changes drastically
at $p=1/2$. Namely, for every fixed $\eps>0$ and every integer $k\ge 2$, if
$p\le 1/2-\eps$, then with high probability the graph of $P^d_p$ has
edge-expansion $\Omega(d^k)$. Thus, random $0/1$ polytopes exhibit an interesting
phase transition at $p=1/2$.
\end{abstract}

\section{Introduction}
A \emph{$0/1$ polytope} in $\mathbb{R}^d$ is the convex hull of a subset of the
hypercube $\{0,1\}^d$.  Equivalently, every vertex of such a polytope has all
its coordinates equal to zero or one.  These polytopes are central objects in
polyhedral combinatorics: many combinatorial structures can be encoded as
characteristic $0/1$ vectors, and the corresponding polytope serves as a
geometric handle for the underlying combinatorial problem via linear
programming~\cite{Schrijver}.  Prominent examples include the perfect matching
polytope, the matroid base polytope, the cut polytope, the stable set polytope,
and the order-ideal polytope; we refer to~\cite{Schrijver} for an extensive
account of their role in combinatorial optimisation.
 
The \emph{graph} $G_P$ of a polytope $P$ has vertices (resp.\ edges) given by
the set of vertices (resp.\ edges) of $P$.
The graph of a polytope $P$ is also known as its \emph{$1$-skeleton}: it contains all faces of $P$ of dimension at most 1.
A basic measure of connectivity of $G_P$ is its
\emph{edge-expansion} (Cheeger constant) defined by 
\[
  h(G_P) \;=\; \min
  \bigg\{\frac{e(S,\, V(G_P)\setminus S)}{|S|}: S \subseteq V(G_P), 0 < |S| \le \frac{|V(G_P)|}{2}\bigg\},
\]
where $e(S, V(G_P)\setminus S)$ denotes the number of edges with one endpoint
in $S$ and the other outside $S$.
 
A primary motivation for providing lower bounds for $h(G_P)$ is its tight connection to the mixing time of the lazy random walk on $G_P$: 
indeed, such bounds readily result in upper bounds on the relaxation time (and hence, on the mixing time) via the discrete Cheeger inequality, see~\cite[Sections 12-13]{LP17}.
Random walks on $0/1$ polytopes are of significant algorithmic interest: they provide a standard method for
sampling an (approximately) uniformly random combinatorial object from the set encoded by the
vertices of that polytope. This approach was used, for example, by Jerrum, Sinclair, and Vigoda~\cite{JSV} to obtain a polynomial-time approximation
scheme for the permanent of a non-negative matrix. Beyond sampling, edge
expansion is intimately connected to several other fundamental properties of
graphs, including connectivity, diameter, and cycle lengths; we refer the reader
to the surveys~\cite{HLW, Krivelevich} for a broad perspective on expansion in
graphs.

While it has long been known that high-dimensional hypercubes have edge-expansion $1$ (following directly from Harper's inequality~\cite{Harper}), $0/1$ polytopes considered in the literature all seem to expand at least equally well.
This motivated Mihail and Vazirani~\cite{Mihail92} to
conjecture that this actually holds for all $0/1$ polytopes.
 
\begin{conjecture}[Mihail--Vazirani, \cite{Mihail92}]\label{conj:MV}
  Every $0/1$ polytope $P$ satisfies $h(G_P) \ge 1$.
\end{conjecture}
 
A proof of Conjecture~\ref{conj:MV} would have far-reaching algorithmic
consequences~\cite{MV92,Gillmann,Kaibel04}.
By now, it has been verified for several interesting families.
Kaibel~\cite{Kaibel04} proved it for all $0/1$ polytopes of dimension at most
five, all simple $0/1$ polytopes, all hypersimplices, all stable set polytopes,
and all perfect matching polytopes.  The conjecture was also established for
matching polytopes, order-ideal polytopes, and independent-set polytopes
in~\cite{Mihailthesis}, and for balanced matroid base polytopes
in~\cite{FM92}.  A recent major breakthrough is the result of Anari, Liu, Oveis
Gharan, and Vinzant~\cite{ALOV}, who proved the conjecture for the matroid base
polytope of \emph{any} matroid using the theory of log-concave polynomials and
high-dimensional expanders. As a consequence, they obtained a fully polynomial
randomised approximation scheme (FPRAS) for counting the number of bases of any
matroid, resolving a long-standing open problem in approximate counting. Despite this progress, the conjecture remains wide
open in general; in particular, it is unknown for the cut polytope, the
travelling-salesman polytope, and most other polytopes arising in combinatorial
optimisation.
 
A natural important setting for Conjecture~\ref{conj:MV}, and one that captures the
``typical'' behaviour of $0/1$ polytopes, is the \emph{random $0/1$ polytope}
$P^d_p = \operatorname{conv}(Q_p^d)$, where $Q_p^d \subseteq \{0,1\}^d$ is
obtained by retaining each vertex of $\{0,1\}^d$ independently with probability
$p \in (0,1)$.  Setting $p = 1/2$ here yields the uniform model.  The study of
random $0/1$ polytopes was initiated by Gillmann~\cite{Gillmann}, and the
question of whether the Mihail--Vazirani conjecture holds in this random setting
was asked explicitly more than twenty years ago by Kaibel and Remshagen~\cite{KR03}.
 
Progress on this question was first made by Leroux and
Rademacher~\cite{LR22}, who proved that, for every $p \in (0,1)$, the graph
of $P^d_p$ has edge expansion at least $1/(12d)$ with high probability (i.e.\
with probability tending to one as $n\to\infty$).  This result was substantially
improved by Ferber, Krivelevich, Sales, and Samotij~\cite{FKSS26}, who showed
that for $p \ge 2^{-0.99d}$ the edge expansion is bounded below by an
\emph{absolute positive constant} with high probability.  At the other extreme,
when $p$ is very small, the problem becomes more tractable: Bondarenko and
Brodskiy~\cite{BB} showed that for $p \le 2^{-5d/6}$ the polytope graph is a
complete graph with high probability (so its expansion is as large as possible),
and a short combinatorial proof of a related threshold phenomenon was given
in~\cite{BEF25}. Specifically, \cite{BEF25} shows that the graph of $P^d_p$ is typically sparse when 
$\log(1/p)\le (1/2-\eps)d$ and becomes dense when $\log(1/p)\ge (1/2+\eps)d$, 
and a similar ``logarithmic'' threshold for the graph being a clique is exhibited around $p = 2^{-\delta d}$ for $\delta \approx 0.8295$.
 
In this paper, we resolve the Mihail--Vazirani conjecture in the
random $0/1$ polytope model. 
Our first main result confirms the conjecture in a very strong form.
 
\begin{theorem}\label{thm:main1}
 For every $\varepsilon > 0$, there exists $c=c(\eps)>0$ such that, for every $p \in (0,1-\eps]$, whp\footnote{With high probability, that is, with probability tending to one as $d$ tends to infinity.}
  \[
    h(G_{P^d_p}) \;\ge\; cd.
  \]
\end{theorem}
 
For every $p$ bounded away from $1$, Theorem~\ref{thm:main1} shows that the edge-expansion grows linearly
in the dimension $d$: a factor of $d$ stronger than the conjecture itself
requires. Observe that this bound is tight up to the value of the constant $c$ whenever $p>1/2$: 
indeed, a simple second moment calculation shows that $Q^d_p$ retains some vertex $v$ together with all its $d$ neighbours in $Q^d$,
and thus the neighbourhoods of $v$ in $Q^d$ and in the graph of $P^d_p$ coincide.
This naturally raises the question whether the lower bound can be improved when $p<1/2$. 
Our second result answers this question in the positive by exhibiting a striking \emph{phase transition} at $p = 1/2$.
 
\begin{theorem}\label{thm:main2}
For every $\eps > 0$ and every integer $k \ge 2$, there exists $c_k=c_k(\eps)>0$ such that, for every $p = p(d)$ with $p\in \left[2^{-0.83d},1/2-\eps\right]$, whp
  \[
    h(G_{P^d_p}) \;\ge\; c_k d^k.
  \]
\end{theorem}
Several comments are in place. As mentioned before, by \cite{BEF25}, when $p\le 2^{-0.83d}$, we have that $P^d_p$ is a clique with high probability. We note here that the key new ingredient is a coupling between
the long edges of the random polytope and mixed percolation on
hypercubes. Here, mixed percolation means that vertices and edges are retained
independently, with probabilities $p$ and $q$,
respectively; we denote the resulting random subgraph of the
hypercube by $Q^d_{p,q}$. This defines a natural random graph
model on the hypercube which may be of independent interest. For a fixed odd distance $k$, let
$Q^d(k)$ denote the graph with vertex set $V(Q^d)$ in which two
vertices $u,v$ are adjacent if they differ in exactly $k$
coordinates. We exhibit many $m$-dimensional cubes inside
$Q^d(k)$ such that the restriction of the random graph to each
such cube has the law of $Q^m_{p,q}$; see
Lemma~\ref{lem:reduction}. This reduces the geometric problem of
controlling long edges in the random polytope to a local
expansion statement for a locally i.i.d.\ model, proved in
Proposition~\ref{prop: expansion of high degree}. Averaging these
local expansion estimates over a large overlapping family of
cubes, and then applying a renormalisation step together with
Lemma~\ref{lem:normalization}, allows us to transfer the resulting
boundary edges back to $P_p^d$.

Let us further remark that, for Theorem~\ref{thm:main2}, essentially the same argument applies when $k=\Omega(\log \log d)$, which is optimal up to the value of the implicit constant for all fixed $p<1/2$. We nevertheless present the proof for arbitrary $k$ for the sake of transparency and readability. The only required modification is to allow the parameters $q,\alpha$, and $\gamma$ in Proposition~\ref{prop: expansion of high degree} to tend to zero (specifically, with $q,\alpha,\gamma=d^{-o(1)}$). With this adjustment, the proof of the proposition --- and consequently of the theorem --— goes through almost verbatim, with only minor changes in the application of Lemma~\ref{lem: decent vertices}. We refer the reader to the statement of Lemma~\ref{lem:min degree} to see explicitly where and why the condition $k=\Omega(\log \log d)$ arises as a bottleneck in the argument.


The paper is structured as follows. In Section \ref{sec: outline} we present an outline of the proof of the main theorems, and in Section \ref{sec: prelim} we collect relevant notation and several useful lemmas. The proof itself has three layers. First, in Section~\ref{subsec:renorm} we renormalise the sparse regime to a density bounded away from zero and show, via Lemma~\ref{lem:normalization}, that expansion in the projected model lifts back to the original polytope. Second, in Section~\ref{sec: mixed perco} we prove Proposition~\ref{prop: expansion of high degree}, an expansion theorem for mixed vertex-edge percolation on the hypercube, which could be of independent interest. Third, and this is the key new structural input, in Section~\ref{section: family of cubes} we show that inside a carefully chosen family of subcubes of $Q^d(k)$ (see Section \ref{sec: prelim} for notation), the long-edge graph of the random polytope is distributed as such a mixed-percolated hypercube. Section~\ref{sec: long distances} then bootstraps these local expansion estimates into a global lower bound on the polytope boundary, from which Theorems~\ref{thm:main1} and~\ref{thm:main2} follow in Section \ref{sec: proof of Theorems}.

\bigskip
\noindent
{\bf Note added in proof.} While this paper was being prepared, we
learned that essentially the same results were obtained independently and
simultaneously by Guo and Tomon~\cite{GT26}, using very different
methods.

\section{Proof outline}\label{sec: outline}
In this section, we present an outline of the proof of Theorems~\ref{thm:main1} and~\ref{thm:main2}. 
We focus on Theorem~\ref{thm:main2}; at the end of the section, we explain how to derive Theorem~\ref{thm:main1} with minor modifications of the same ideas.

The first step is a renormalisation argument (used in~\cite{FKSS26}) whose purpose is to replace the original model by a coarser one where the relevant density parameter is of constant order.
Concretely, we partition the hypercube into $b$-dimensional fibres and collapse each fibre to a single vertex. A fibre is occupied if it contains at least one surviving vertex, so the projected model has effective density $\rho = 1-(1-p)^{2^b}$. 
The point of this reduction is that, after renormalisation, the density of surviving vertices belongs to a range where certain local expansion estimates, which we develop, can be applied; 
this local expansion is lifted from the projected polytope to the original one via Lemma~\ref{lem:normalization}. 

After this reduction, the proof of Theorem~\ref{thm:main2} has three conceptual steps. The first is a local expansion result for mixed percolation on the hypercube $Q^d_{p,q}$ (Proposition~\ref{prop: expansion of high degree}, which may be of independent interest).
We show that whp any set whose vertices all have linear degree in $Q^d_{p,q}$ must have neighbourhood of order $\Omega(|S|/\sqrt d)$. 
The proof is a thinning-and-sprinkling argument. If such a set $S$ existed in $Q^d_{p,q}$, by moving to $Q^d_{p/2,q}$ via vertex-thinning, it is not too unlikely that this set $S$ has an empty neighbourhood in $Q^d_{p/2,q}$, that is, it becomes disconnected from the rest of the graph.
Moreover, one can show that typically it still retains many vertices of high degree (that is, proportional to $d$). 
We show that having a set with many vertices of high degree and empty neighbourhood is very unlikely via a partitioning argument (in Lemma~\ref{lem: decent vertices}) and a delicate structural analysis of $Q^d_{p,q}$ (in Lemmas~\ref{lem: expanding matching}--\ref{lem: sprinkling event}).

The second step is to connect this local hypercube statement to the geometry of the random polytope. 
The expansion we seek in $P^d_p$ is not coming from the ordinary hypercube edges but from longer ones, created by the convex hull structure. 
These long edges are difficult to control directly as they are globally dependent. 
The key idea is to consider them inside carefully chosen subgraphs. 
For a fixed odd distance $k$, Section~\ref{section: family of cubes} constructs a family $\cQ_{k,m}$ of $m$-dimensional cubes inside the graph $Q^d(k)$ with vertices $V(Q^d)$ and an edge between any two vertices at Hamming distance exactly $k$.
Each cube $H\in \cQ_{k,m}$ is a small subgraph of the long-edge graph with the crucial property that
$H\cap P^d_p$ has distribution generated from mixed percolation on $Q^m$
(which is the content of Lemma~\ref{lem:reduction}). 
This explains the relevance of our mixed-percolation estimates: they give a lower bound on the edge-expansion on each subcube.

The third step is the global amplification, carried out in Section~\ref{sec: long distances}. The cubes in $\cQ_{k,m}$ are highly overlapping, so a single set $S$ is seen many times from many different local viewpoints. 
Our purpose is to average the local expansion contributions over the whole family of cubes (this is done in Lemma~\ref{lem: bootstrapped expansion}).
For this, we need to ensure that many cubes have a moderate proportion of their vertices occupied by $S$, as this is the regime where our expansion estimates are applicable and efficient.
The remaining lemmas in Section~\ref{sec: long distances} ensure this fact and show that no edge from $S$ to its complement is counted too many times in the averaging procedure.
Together these conclusions ensure that the local $\Omega(|S|/\sqrt d)$ expansion inside cubes is amplified to the desired global expansion bound for $P_p^d$.

The proof of Theorem~\ref{thm:main1} follows the same general strategy with two changes. 
First, thanks to Theorem~\ref{thm:main2}, we can assume that $p\geq 1/3$, so the renormalisation step in the proof of Theorem~\ref{thm:main2} is not needed.
However, $P^d_p$ may now contain a few vertices which preserve their neighbourhood in $Q^d$. 
To take this phenomenon into account, we use an expansion estimate for dense cubes from~\cite{FKSS26} (Lemma~\ref{lem: Wojtek}) together with the same normalisation as in Theorem~\ref{thm:main2}, albeit with different parameters. 
Combined with the prior ``long edges'' argument, this gives the lower bound $h(G_{P_p^d})=\Omega(d)$.

\section{Preliminaries}\label{sec: prelim}
\subsection{Notation}
For real numbers $a,x$ and $b>0$, we write $x=a\pm b$ to say that $x\in [a-b,a+b]$. We denote by $\mathrm{Bin}(n,p)$ the binomial distribution with parameters $n$ and $p$.
In the $d$-dimensional binary hypercube $Q^d$, the Hamming distance between two vertices corresponds to the number of bits where they differ.
Rounding is ignored when insignificant for our arguments.

Given a graph $G=(V,E)$ and disjoint sets of vertices $A,B\subseteq V$, we write $e_G(A,B)$ for the number of edges with one endpoint in $A$ and another endpoint in $B$.
We sometimes write $e_P(A,B)$ for a polytope $P$: in this case, the underlying graph is the graph of the polytope $P$.
Given graphs $G_1\subseteq G_2$ and sets $U\subseteq V(G_2)$ and $F\subseteq E(G_2)$, we write $G_1[U]$ for the subgraph of $G_1$ induced by $U\cap V(G_1)$, and $G_1[F]$ for the spanning subgraph of $G_1$ with edge set $F\cap E(G_1)$.
Further, for $G=(V,E)$ and a parameter $p$, we write $E_p\subseteq E$ for a random subset containing every $e\in E$ independently and with probability $p$. 
Similarly, we write $V_p\subseteq V$ for a random subset containing every $v\in V$ independently and with probability $p$. 
Given parameters $p,q\in[0,1]$, we write $G_{p,q}\coloneqq G(V_p,E_q)$, that is, the induced subgraph obtained by retaining every vertex probability $p$ and every edge with probability $q$. 
For $F\subseteq E$ and $u\in V$, we write $N_{F}(v)$ to denote the neighbours of $v$ in $(V,F)$, and $N_F[v] = N_F(v)\cup \{v\}$. This notation readily generalises for larger sets $S\subseteq V(G)$ by setting $N_F(S) = (\bigcup_{v\in S} N_F(v))\setminus S$ and $N_F[S] = N_F(S)\cup S$.
Throughout the text, unless explicitly stated otherwise, we assume that $V\coloneqq V(Q^d)$ and $E\coloneqq E(Q^d)$; the subscript in $e_G(\cdot,\cdot)$ and $N_E(\cdot)$ are omitted in this case.
For vertices $x,y\in V(Q^d)$, we denote by $Q^d[x,y]$ the smallest subcube of $Q^d$ containing each of $x$ and $y$. 
Given a random $0/1$ polytope $P^d_p$, for every $j\in \{1,\ldots,d\}$, we write $P_j\coloneqq P^d_p(j)$ for the subset of edges $uv$ of $P^d_p$ whose endpoints differ in $j$ coordinates, i.e., have Hamming distance $j$. 
For $k\ge 1$, we denote by $Q^d(k)$ the graph whose vertex set is $V(Q^d)$ and edges go between pairs of vertices $u,v$ differing in exactly $k$ coordinates.

\subsection{Auxiliary results}
We begin with Harper's edge-isoperimetric inequality for hypercubes, see~\cite{H64} and also \cite{B67,H76,L64}.
\begin{lemma}\label{l: Harper edge}
For every integer $d \ge 1$ and for every set $S\subseteq V(Q^d)$,
\begin{align*}
e(S,V(Q^d)\setminus S)\ge \left(d-\log_2|S|\right) |S|.
\end{align*}
\end{lemma}

We will also use the following (crude) immediate corollary of Harper's vertex-isoperimetric inequality~\cite{H66}.
\begin{lemma}\label{l: harper vertex}
Fix $\eps\in (0,1)$. For every $|S|\le (1-\eps)\cdot 2^{d}$, we have $|N(S)|\ge \eps|S|/(10\sqrt{d})$. 
\end{lemma}
\begin{remark}
    We note that, for small $\eps>0$, the correct dependency of the lower bound on $\eps$ is of the form $\eps\sqrt{\log(1/\eps)}$, but the above (simpler and cruder) bound will suffice for our needs.
\end{remark}

We also use the classic Chernoff-type tail bound on the binomial distribution (see e.g.~Appendix~A in \cite{AS16}).
\begin{lemma}\label{l: chernoff}
Consider $d\in \mathbb{N}$, $p\in [0,1]$ and $X\sim \mathrm{Bin}(d,p)$. Then, for every $t\in (0,dp/2]$, 
\begin{align*}
    &\mathbb{P}\left[X\neq dp\pm t\right]\le 2\exp\left(-\frac{t^2}{3dp}\right).
\end{align*}
\end{lemma}

We also utilise a variant of the bounded difference inequality (see \cite{War16} and also~Chapter 7 in \cite{AS16}).
\begin{lemma}\label{l: azuma}
Fix $m\in \mathbb{N}$ and let $p\in [0,1]$. Consider a random vector $X = (X_1,X_2,\ldots, X_m)$ with independent entries distributed according to $\mathrm{Bernoulli}(p)$. 
Fix $\Lambda = \{0,1\}^m$ and a function $f:\Lambda\to\mathbb{R}$ such that there is $k_1,\ldots, k_m \in \mathbb{R}$ with the property that, for every $x,x' \in \Lambda$ which differ only in the $i$-th coordinate, $|f(x)-f(x')|\le k_i$ for every $i\in \{1,\ldots, m\}$. 
Then, for every $t\ge 0$,
\begin{align*}
    \mathbb{P}\left[\big|f(X)-\mathbb{E}\left[f(X)\right]\big|\ge t\right]\le 2\exp\left(-\frac{t^2}{p\cdot 2\sum_{i=1}^mk_i^2+2t/3\cdot \max_ik_i}\right).
\end{align*}
\end{lemma}

We also need the following sufficient condition for the existence of a ``long edge'' in a polytope, see \cite[Proposition 2.6]{FKSS26}.

\begin{lemma}\label{lem: poly long edge exists}
Fix positive integers $k\le d$ and a polytope $P$ which is the convex hull of some subset of $V(Q^d)$.
Fix a $k$-dimensional face $F$ of $Q^d$, and two vertices $x,y\in P\cap F$ at Euclidean distance $\sqrt{k}$ from each other. 
If there is a $(k-1)$-dimensional face $F'\subseteq F$ such that $P\cap F'=\{x\}$, then $xy$ is an edge in $P$.
\end{lemma}

\begin{remark}\label{remark: edge via cube}
By Lemma~\ref{lem: poly long edge exists} for $P^d_p$, for every $u,v\in V(Q^d)$ such that $V(Q^d_p[u,v])=\{u,v\}$, $uv$ is an edge in $P^d_p$.
More generally, the event that $uv$ is an edge in $P^d_p$ only depends on the states of the vertices in $Q^d[u,v]$: indeed, for every closed half-space $H$ containing $Q^d$ with $\partial H\cap Q^d=Q^d[u,v]$, 
we have (by e.g.\ \cite[Claim 2.8]{FKSS26})
\[P^d_p\cap \partial H = \mathrm{conv}(Q^d_p)\cap \partial H = \mathrm{conv}(Q^d_p\cap \partial H).\]
\end{remark}

Finally, we will also use the following \cite[Theorem 4.1]{FKSS26}.

\begin{lemma}\label{lem: Wojtek}
    Let $p\in[1-e^{-4},1]$. With high probability for every $S\subseteq V(P_p^{d})$ with $|S|\leq 3/4\cdot 2^d$ it holds that 
    \[e_{P_p^d}(S,S^c)\geq |S|/8\cdot \log_2(2^d/|S|).\]
\end{lemma}

\section{Renormalisation}\label{subsec:renorm}
This section provides the renormalisation step. We partition $Q^d$ into $b$-dimensional fibres and choose $b$ so that a fibre is occupied with probability $\rho=1-(1-p)^{2^b}$, where $b$ is chosen such that the projected model lives at the regime where the mixed-percolation input from later sections will apply. The key point is that expansion in the projected polytope can be lifted back to the original polytope: Lemma~\ref{lem:normalization} shows that every projected edge produces a genuine edge of $P_p^d$.

Partition the hypercube $Q^d = Q^{d-b}\times Q^b$ into $2^{d-b}$ subcubes of dimension $b$ according to the projection of each vertex on the first $d-b$ coordinates: in particular, the vertices of $Q^{d-b}$ are identified with the $b$-dimensional cubes in the partition.
We define $Q\subseteq Q^{d-b}$ to be the induced subgraph obtained by keeping only the vertices whose corresponding cube of dimension $b$ in the above partition of $Q^d$ contains at least one vertex present in $Q^d_p$. 
Note that $Q$ is a random subgraph of $Q^{d-b}$ obtained after vertex percolation with parameter $\rho := 1-(1-p)^{2^b}$; in particular, we see $P^{d-b}_{\rho}$ as the convex hull of the vertices of $Q$. 
Further, we define a projection map $\pi := \pi_{b,d}: \mathbb R^d\to \mathbb R^{d-b}$ by setting 
\[\pi((a_1,\ldots,a_d)) = (a_1,\ldots,a_{d-b}).\]
\begin{lemma}\label{lem:normalization}
For every edge $uv\in E(P^{d-b}_\rho)$ and every vertex $w\in \pi^{-1}(u)$, $P^d_p$ contains an edge between $w$ and some vertex in $\pi^{-1}(v)$.
\end{lemma}
\begin{proof}
As $uv\in E(P^{d-b}_{\rho})$, there is a vector $\mathbf{v}=(\mathbf{v}_1,\ldots,\mathbf{v}_{d-b})\in \mathbb R^{d-b}$ and a number $c$ such that
\[\langle \mathbf{v},u\rangle = \langle \mathbf{v},v\rangle = c\qquad \text{and}\qquad \forall z\in V(P^{d-b}_\rho)\setminus \{u,v\},\; \langle \mathbf{v},z\rangle > c.\]
Define the vector $\mathbf{v}_0 = (\mathbf{v}_1,\ldots,\mathbf{v}_{d-b},0,\ldots,0)\in \mathbb R^d$ and denote by $H_c$ the $(d-1)$-dimensional hyperplane $\{w\in \mathbb R^d: \langle \mathbf{v}_0,w\rangle = c\}$.
Then, by the choice of $\mathbf{v}$, 
\[R := \pi^{-1}(\{u,v\})\subset H_c\qquad \text{and}\qquad \pi^{-1}(V(P^{d-b}_\rho)\setminus \{u,v\})\cap H_c=\emptyset.\]
Further, note that $R \supseteq R_p := V(P^d_p)\cap H_c$ and that
$R$ is isomorphic to $Q^{b+1}$. Indeed, $\pi^{-1}(u)$ and
$\pi^{-1}(v)$ are $b$-dimensional subcubes whose first $d-b$
coordinates coincide with the ones of $u$ and $v$, respectively. Since $u$ and
$v$ form an edge in the convex hull, (possibly after affine transformation) these two subcubes together
form a $(b+1)$-dimensional subcube. Thus, somewhat abusively, we
identify $R$ with $Q^{b+1}$ throughout the rest of the proof.

To conclude, it suffices to show that $w$ is contained in at least one edge in the 1-skeleton of the polytope $\mathrm{conv}(R_p)$.
To this end, denote by $w'\in V(P^d_p)\cap \pi^{-1}(v)$ a vertex such that $w'$ is the closest vertex to $w$ for the Hamming distance in the cube $R$ which percolates in $R_p$.
Then, the smallest subcube $R[w,w']$ of $R$ containing each of $w,w'$ contains the face 
\[F = R[w,w']\cap \pi^{-1}(v)\]
where each vertex in $F\setminus \{w'\}$ is closer to $w$ than $w'$ for the Hamming distance.
Thus, $F\cap V(R[w,w'])=\{w'\}$, which is enough to conclude that $ww'$ is an edge in the graph of $P^d_p$ by Lemma \ref{lem: poly long edge exists}.
\end{proof}

We will also make use of the following lemma.
\begin{lemma}\label{lem: degrees}
Fix $p\in (0,1-\varepsilon]$. Then, there exists $\alpha>0$ such that the following hold with high probability.
    \begin{itemize}
        \item If $p\geq 1/2-\varepsilon$, then the number of vertices in $Q_p^d$ with degree at least $(1-\alpha)d$ is at most $2^{(1-\alpha^3) d}$.
        \item If $p\leq 1/2-\varepsilon$, then every vertex in $Q_p^d$ has degree at most $(1-\alpha)d$.
    \end{itemize}
\end{lemma}
\begin{proof}
We begin with the first item. By Lemma \ref{l: chernoff}, the probability a vertex has degree at least $(1-\alpha)d$ is at most
\begin{align*}
    \mathbb P\left[\mathrm{Bin}(d,1-\eps)\ge (1-\alpha) d\right]\le \exp\left(-(\eps-\alpha)^2d/3\right).
\end{align*}
Thus, the expected number of such vertices is at most $2^{d-(\eps-\alpha)^2d/3}$. Therefore, by Markov's inequality, with high probability there are at most $2^{(1-\alpha^3)d}$ such vertices, where we choose $\alpha>0$ sufficiently small with respect to $\eps$.

For the second item, 
the probability a vertex has degree at least $(1-\alpha)d$ is at most
\begin{align*}
    \mathbb P\left[\mathrm{Bin}(d,1/2-\eps)\ge (1-\alpha) d\right]\le \sum_{i=(1-\alpha)d}^d\binom{d}{i}(1/2-\eps)^i\le \left[\left(2e/\alpha\right)^{\alpha}\cdot (1/2-\eps)^{(1-\alpha)}\right]^d.
\end{align*}
By choosing $\alpha>0$ small enough with repsect to $\eps$, the base of the above exponential is strictly less than $1/2$. Then, taking the union bound over the at most $2^d$ vertices of the hypercube completes the proof.
\end{proof}

\section{Expansion in the mixed-percolated hypercube}\label{sec: mixed perco}
This section contains the local random-graph input for the rest of the paper. 
We study mixed percolation on the hypercube, where vertices and
edges are retained independently, and prove that any set whose
vertices all have degree linear in $d$ has external
vertex-neighbourhood of order $\Omega(|S|/\sqrt d)$.
In particular, the goal of this section is to prove the following proposition, which may be of independent interest.

\begin{proposition}\label{prop: expansion of high degree}
Fix $p,q,\alpha\in (0,1]$ and sufficiently small positive constants $\eta\coloneqq \eta(p,q,\alpha)$ and $\gamma\coloneqq \gamma(\eta)$. 
The following holds with probability at least $1-2^{-\eta d^{3/2}}$: for every set $S\subseteq V(Q^d_{p,q})$ with $|S|\le 3|V(Q^d_{p,q})|/4$ and such that every $v\in S$ has degree at least $\alpha d$ in $Q^d_{p,q}$, we have $|N_{Q^d_{p,q}}(S)|\ge \gamma|S|/\sqrt{d}$.
\end{proposition}
\noindent
Note that this bound is tight since the set of vertices below the middle layer(s) expands by a factor of~$\Theta(1/\sqrt{d})$. 

The proof of Proposition~\ref{prop: expansion of high degree} is split into two regimes.  
For very small sets, namely $|S|\le d^2$, the statement is essentially a counting argument. Since the induced subgraph on $S$ cannot contain too many edges by Harper's inequality, while every vertex of $S$ has degree at least $\alpha d$, most of the edges incident to $S$ must leave $S$. 
This already gives a much stronger lower bound on the size of the neighbourhood of $S$ than the one claimed. 

The interesting case is when $|S|$ is larger. The proof then follows a thinning--and--sprinkling argument, in the spirit of \cite{ADLZ25}. Recall the sets $V_p, E_q$ from the notation section. Observe that if there was a set $S$ with degrees at least $\alpha d$ which did not expand well in $Q^d_{p,q}$, then when moving to, say, $Q^d_{p/2,q}$, it is not too unlikely that this set $S$ has empty neighbourhood in $Q^d_{p/2,q}$, that is, it is disconnected from the rest of the graph there. On the other hand, one can show that typically, it still retains many vertices of high degree (that is, proportional to $d$). Now, fix some small density $\eps>0$, and choose $p', q'$ such that $(1-p')(1-\eps)=1-p/2$ and $(1-q')(1-\eps)=1-q$, repsectively. This way, $V_{p'}\cup V_{\eps}$ has the same distribution as $V_{p/2}$ and $E_{q'}\cup E_{\eps}$ has the same distribution as $E_q$. In Lemma~\ref{lem: decent vertices}, we show that sprinkling by $(V_\eps,E_\eps)$ does not typically create many new vertices of high degree inside a fixed candidate set. 
This lets us transfer any hypothetical bad set in $Q^d_{p,q}$ to a comparable bad set in the partially exposed graph $Q^d_{p',q'}$, while retaining most of the vertices with degree $\Theta(d)$. Next, in Lemma~\ref{lem: expanding matching}, we show that any such set, provided it consists mostly of high-degree vertices, must have a large matching leaving the neighbourhood of its large components. This is the input that turns `few external neighbours' into a concrete obstruction: after sprinkling, a positive fraction of this boundary matching survives, so the set cannot remain disconnected. Finally, Lemma~\ref{lem: sprinkling event} combines these two statements. It shows that the event that a large set of mostly high-degree vertices is disconnected after sprinkling is extremely unlikely. This is the key probabilistic step that rules out bad sets in the range $|S|>d^2$.

We begin by showing that, with suitably high probability, only few vertices in $Q^d$ are adjacent to many vertices in $V_\eps$ or incident to many edges in $E_\eps$.

\begin{lemma}\label{lem: decent vertices}
Fix $\eps,\alpha\in (0,1)$ with $\eps \le 8^{-8/\alpha^2}$ and consider an integer $s\in \{d^2,\ldots, 2^d\}$. 
Then, the probability that there are at least $s$ vertices in $Q^d=(V,E)$, each with at least $\alpha d$ neighbours in $V_\eps$ or incident to at least $\alpha d$ edges in $E_{\eps}$, is at most $2^{-s}$.
\end{lemma}
\begin{proof}
We say that a set $U$ has the $\alpha$-star property if there are $|U|$ vertex-disjoint stars in $Q^d$ with centres in $U$ and each with $(1-\alpha/2)d$ leaves.
We show that, for every set $W'\subseteq V(Q^d)$ with $|W'|\ge s$, there exists  $W\subseteq W'$ with $|W| = \alpha s/(4d)$ having the $\alpha$-star property.
Indeed, fix a maximal subset $U\subseteq W'$ with this property and denote by $F_U$ the corresponding star forest. 
Suppose for contradiction that $|U|<\alpha s/(4d)$. Hence, $|V(F_U)|\leq \left((1-\alpha/2)d+1\right)|U|<\alpha s/4$.
        Further, by maximality of $U$, every vertex $v\in W'\setminus V(F_U)$ has at most $(1-\alpha/2)d$ neighbours in $V(Q^d)\setminus V(F_U)$. Thus, 
    \begin{align*}
     e(W'\setminus V(F_U), V(F_U))\ge |W'\setminus V(F_U)|\alpha d/2\ge (|W'|-|V(F_U)|)\alpha d/2\ge (1-\alpha/4)\alpha sd/2.
    \end{align*}
    On the other hand, $F_U$ is incident to at most $|V(F_U)|d\le |U|d^2$ edges, and therefore
    \begin{align*}
        |U|d^2\ge (1-\alpha/4)\alpha sd/2\implies |U|\ge \frac{(1-\alpha/4)\alpha s}{2d}\ge \frac{\alpha s}{4d},
    \end{align*}
    a contradiction.  

    Now, fix $W$ and $F_W$ as above, and set $X_W\coloneqq|V(F_W)\cap V_\eps|+|F_W\cap E_\eps|$. Since $\frac{\alpha s}{4d}+2\left(1-\frac{\alpha}{2}\right)d\cdot\frac{\alpha s}{4d}< \alpha s$, we have that $X_W$ is dominated by $\mathrm{Bin}(\alpha s,\varepsilon)$. Hence,
    \begin{align*}
    \mathbb P[X_W\geq \alpha^2 s/8]\leq\sum_{i=\alpha^2 s/8}^{\alpha s}\binom{\alpha s}{i}\varepsilon^{i}\leq 2^{\alpha s}\cdot \varepsilon^{\alpha^2s/8}\leq(2\eps^{\alpha^2/8})^{s}\leq 4^{-s},
    \end{align*}
    where for the last inequality we used that $\eps\le 8^{-8/\alpha^2}$. Since there are at most $(2^d)^{\alpha s/(4d)}$ choices of $W$, by a union bound, the probability there exists $W$ with $X_W\geq \alpha^2 s/8$ is at most $(2^d)^{\alpha s/(4d)}\cdot 4^{-s}\le 2^{-s}$.

    Finally, assume that there are more than $s$ vertices with at least $\alpha d$ neighbours in $V_\eps$ (in $Q^d$) or incident to at least $\alpha d$ edges in $E_{\eps}$, and denote this set of vertices by $W'$. 
    Thus, there exists a subset $W\subseteq W'$ of size $\alpha s/(4d)$ having the $\alpha$-star property (so $F_W$ is well-defined). Then, by assumption on $W'$, we have that 
    $$X_W=|V(F_W)\cap V_{\eps}|+|F_W\cap E_{\eps}|\ge \frac{\alpha d}{2}\cdot |W|=\frac{\alpha^2s}{8},$$ which happens with probability at most $2^{-s}$, as required.
\end{proof}


In the next lemma we fixed a a moderate size family of large subsets of the hypercube and 
show that, upon (vertex- and edge-)sprinkling, with suitably high probability, for every set in the family one can typically find a large matching leaving it.


\begin{lemma}\label{lem: expanding matching}
Fix $\eps\in (0,1)$, an integer $s\in \{d^2,\ldots,2^d\}$, sets $U\subseteq V(Q^d)$, $F\subseteq E(Q^d)$ and denote by $\cC_l$ the set of components of size at least $d^2$ in $(U,F)$. With every family of components $\cC'\subseteq \cC_l$ associate a set $A(\cC')\subseteq V(Q^d)$ such that $|A(\cC')|\in [s/2,(1-\eps)2^d]$. 
Then, with probability at least $1-2^{-\eps^4s/(980\sqrt{d})}$, for every $\cC'\subseteq \cC_l$ with $|V(\cC')|\le s$, the graph $(V_{\eps},E_{\eps})$ contains a matching size at least $\eps^4s/(160\sqrt{d})$ leaving $A(\cC')$.
\end{lemma} 

\begin{proof}
Set $A\coloneqq A(\cC')$,
$B\coloneqq V(Q^d)\setminus A$ and let $M$ be a maximum matching between $A$ and $B$ in $Q^d$.~The~maximality of $M$ gives $N(A\setminus V(M))\subseteq V(M)$, and by Lemma \ref{l: harper vertex} we have 
$|V(M)|\ge |N(A\setminus V(M))|\ge \eps|A\setminus V(M)|/(10\sqrt{d})$.
Since also $|A|\ge s/2$, either $|V(M)| \ge s/4$ or $|A\setminus V(M)|\ge s/4$, with the latter implying that $|V(M)|\ge \eps s/(40\sqrt{d})$. Overall, $M$ is a matching on at least $\eps s/(80\sqrt{d})$ edges. 

Next, since the number of edges of $M$ in the graph $(V_\eps,E_\eps)$ stochastically dominates $\mathrm{Bin}(\eps s/(80\sqrt{d}),\eps^3)$, by Lemma \ref{l: chernoff}, $\mathbb P[|M\cap E((V_\eps,E_\eps))|\le \eps^4s/(160\sqrt{d})]\le 2e^{-\eps^4s/(960\sqrt{d})}$.
Finally, to choose $\cC'$, we need to choose at most $s/d^2$ components 
among the at most $2^d/d^2$ components in $\cC_l$. Thus, the number of choices for $\cC'$ is at most
\[\sum_{i=1}^{s/d^2}\binom{2^d/d^2}{i}\le \left(3\cdot 2^d/s\right)^{s/d^2}\le 2^{s/d} = 2^{o(s/\sqrt{d})}.\]
A union bound over the number of ways to choose $\cC'$ completes the proof.
\end{proof}

Using the two lemmas above, we show that, with suitably high probability, all vertex sets which span subgraphs of $Q^d_{p,q}$ where almost all vertices have degree $\Theta(d)$ send at least one edge to the rest of the graph.
We say that a vertex set $S$ is \emph{disconnected} in a graph $G$ if $e(S,V(G)\setminus S)=0$.


\begin{lemma}\label{lem: sprinkling event}
Fix $p,q,\alpha\in (0,1)$ and sufficiently small $\eps\coloneqq \eps(p,q,\alpha)>0$ and $\beta\coloneqq \beta(\alpha, \eps)>0$. 
For every $s\in\{d^2,\ldots, 3|V(Q^d_{p,q})|/4\}$, the probability that there exists a disconnected set $S$ on $s$ vertices in $Q^d_{p,q}$ with at most $\beta s/\sqrt{d}$ vertices with degree at most $\alpha d/4$ is bounded from above by $2^{-\eps^5 s/(2\sqrt{d})}$.
\end{lemma}

Before proving the above, we briefly sketch the proof. 
Suppose that $S$ is a disconnected set in $Q^d_{p,q}$ whose vertices mostly have degree at least $\alpha d/4$. 
Using sprinkling, we first reveal a subgraph $(V_{p'},E_{q'})$ and withhold an independent copy of $(V_\eps,E_\eps)$ such that together they reproduce $Q^d_{p,q}$. 
We then decompose $S$ according to its intersection with low-degree vertices in $(V_{p'},E_{q'})$ (a set we call $X$), 
the small components (of size less than $d^2$; this is the family $\cC_s$), vertices outside $V_{p'}$ without edges in $E_{q'}$ towards large components (a set we call $Y$), 
and the large components in $(V_{p'},E_{q'})$ (family $\cC_l$). 
By Lemma~\ref{lem: decent vertices}, only very few vertices of $S$ can lie in $X$, and the contribution of $V(\cC_s)\cup Y$ is also limited, see Figure~\ref{fig:5.4}. 
After removing these parts, we obtain a large subset $S'\subseteq S$ contained in the $E_{q'}$-neighbourhood of $V(\cC_l)$. 
By slightly restricting this set to the closed neighbourhood $A=A(\cC')$ of the family $\cC'$ of large components entirely in $S$,
Lemma~\ref{lem: expanding matching} then guarantees a large matching in $Q^d$ leaving this neighbourhood $A$ when the sprinkling is added back. 
As many of these edges are likely to appear by Lemma~\ref{lem: expanding matching}, this produces many edges from $S'$ to its complement, and some of them land outside $S\setminus S'$, thus contradicting that $S$ is disconnected.

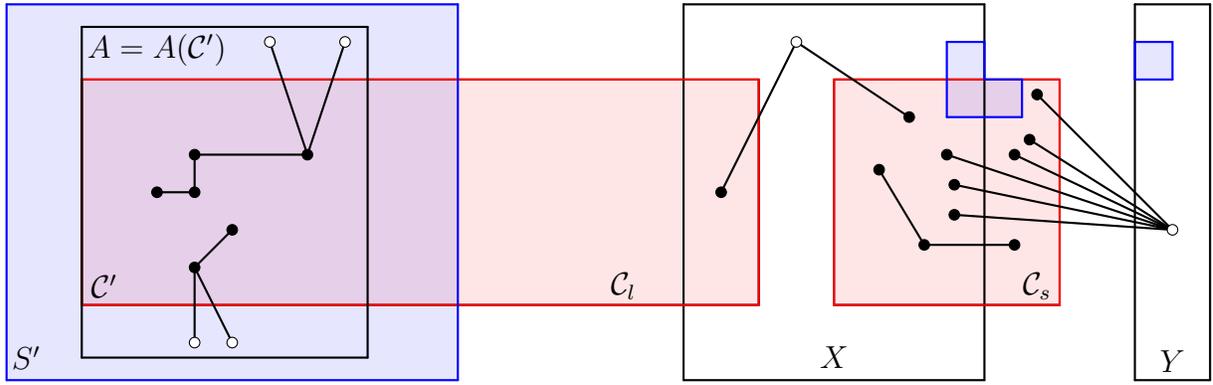
\begin{figure}
\centering
\begin{tikzpicture}[line cap=round,line join=round,x=1cm,y=1cm]
\clip(-8.8,0.7) rectangle (8.1,6.1);
\fill[line width=0.8pt,color=red,fill=red,fill opacity=0.10000000149011612] (-7,5) -- (-7,2) -- (2,2) -- (2,5) -- cycle;
\fill[line width=0.8pt,color=red,fill=red,fill opacity=0.10000000149011612] (3,5) -- (3,2) -- (6,2) -- (6,5) -- cycle;
\fill[line width=0.8pt,color=blue,fill=blue,fill opacity=0.1] (4.5,5.5) -- (4.5,4.5) -- (5.5,4.5) -- (5.5,5) -- (5,5) -- (5,5.5) -- cycle;
\fill[line width=0.8pt,color=blue,fill=blue,fill opacity=0.1] (7,5.5) -- (7.5,5.5) -- (7.5,5) -- (7,5) -- cycle;
\fill[line width=0.8pt,color=blue,fill=blue,fill opacity=0.1] (-2,6) -- (-8,6) -- (-8,1) -- (-2,1) -- cycle;
\draw [line width=0.8pt] (-7,5)-- (-7,2);
\draw [line width=0.8pt] (-7,2)-- (2,2);
\draw [line width=0.8pt] (2,2)-- (2,5);
\draw [line width=0.8pt] (2,5)-- (-7,5);
\draw [line width=0.8pt] (3,5)-- (3,2);
\draw [line width=0.8pt] (3,2)-- (6,2);
\draw [line width=0.8pt] (6,2)-- (6,5);
\draw [line width=0.8pt] (6,5)-- (3,5);
\draw [line width=0.8pt] (1,6)-- (1,1);
\draw [line width=0.8pt] (1,1)-- (5,1);
\draw [line width=0.8pt] (5,1)-- (5,6);
\draw [line width=0.8pt] (5,6)-- (1,6);
\draw [line width=0.8pt] (7,6)-- (7,1);
\draw [line width=0.8pt] (7,1)-- (8,1);
\draw [line width=0.8pt] (8,1)-- (8,6);
\draw [line width=0.8pt] (8,6)-- (7,6);
\draw [line width=0.8pt,color=red] (-7,5)-- (-7,2);
\draw [line width=0.8pt,color=red] (-7,2)-- (2,2);
\draw [line width=0.8pt,color=red] (2,2)-- (2,5);
\draw [line width=0.8pt,color=red] (2,5)-- (-7,5);
\draw [line width=0.8pt,color=red] (3,5)-- (3,2);
\draw [line width=0.8pt,color=red] (3,2)-- (6,2);
\draw [line width=0.8pt,color=red] (6,2)-- (6,5);
\draw [line width=0.8pt,color=red] (6,5)-- (3,5);
\draw [line width=0.8pt,color=blue] (4.5,5.5)-- (4.5,4.5);
\draw [line width=0.8pt,color=blue] (4.5,4.5)-- (5.5,4.5);
\draw [line width=0.8pt,color=blue] (5.5,4.5)-- (5.5,5);
\draw [line width=0.8pt,color=blue] (5.5,5)-- (5,5);
\draw [line width=0.8pt,color=blue] (5,5)-- (5,5.5);
\draw [line width=0.8pt,color=blue] (5,5.5)-- (4.5,5.5);
\draw [line width=0.8pt,color=blue] (7,5.5)-- (7.5,5.5);
\draw [line width=0.8pt,color=blue] (7.5,5.5)-- (7.5,5);
\draw [line width=0.8pt,color=blue] (7.5,5)-- (7,5);
\draw [line width=0.8pt,color=blue] (7,5)-- (7,5.5);
\draw [line width=0.8pt,color=blue] (-2,6)-- (-8,6);
\draw [line width=0.8pt,color=blue] (-8,6)-- (-8,1);
\draw [line width=0.8pt,color=blue] (-8,1)-- (-2,1);
\draw [line width=0.8pt,color=blue] (-2,1)-- (-2,6);
\draw [line width=0.8pt] (7.5,3)-- (5.7,4.8);
\draw [line width=0.8pt] (7.5,3)-- (4.6,3.6);
\draw [line width=0.8pt] (7.5,3)-- (5.4,4);
\draw [line width=0.8pt] (7.5,3)-- (4.6,3.2);
\draw [line width=0.8pt] (7.5,3)-- (5.6,4.2);
\draw [line width=0.8pt] (7.5,3)-- (4.5,4);
\draw [line width=0.8pt] (4.2,2.8)-- (3.6,3.8);
\draw [line width=0.8pt] (4.2,2.8)-- (5.4,2.8);
\draw [line width=0.8pt] (2.5,5.5)-- (1.5,3.5);
\draw [line width=0.8pt] (2.5,5.5)-- (4,4.5);
\draw [line width=0.8pt] (-4,4)-- (-4.5,5.5);
\draw [line width=0.8pt] (-4,4)-- (-3.5,5.5);
\draw [line width=0.8pt] (-4,4)-- (-5.5,4);
\draw [line width=0.8pt] (-5.5,4)-- (-5.5,3.5);
\draw [line width=0.8pt] (-5.5,3.5)-- (-6,3.5);
\draw [line width=0.8pt] (-5.5,2.5)-- (-5.5,1.5);
\draw [line width=0.8pt] (-5.5,2.5)-- (-5,1.5);
\draw [line width=0.8pt] (-5.5,2.5)-- (-5,3);
\begin{scriptsize}
\draw[color=black] (0.2,2.25) node {\large{$\cC_l$}};
\draw[color=black] (-6.7,2.25) node {\large{$\cC'$}};
\draw[color=black] (-6,5.4) node {\large{$A=A(\cC')$}};
\draw[color=black] (5.7,2.25) node {\large{$\cC_s$}};
\draw[color=black] (3,1.3) node {\large{$X$}};
\draw[color=black] (7.5,1.25) node {\large{$Y$}};
\draw[color=black] (-7.737528583463091,1.3) node {\large{$S'$}};
\draw [fill=white] (7.5,3) circle (2pt);
\draw [fill=black] (5.7,4.8) circle (2pt);
\draw [fill=black] (4.6,3.6) circle (2pt);
\draw [fill=black] (5.4,4) circle (2pt);
\draw [fill=black] (4.6,3.2) circle (2pt);
\draw [fill=black] (5.6,4.2) circle (2pt);
\draw [fill=black] (4.5,4) circle (2pt);
\draw [fill=black] (4.2,2.8) circle (2pt);
\draw [fill=black] (3.6,3.8) circle (2pt);
\draw [fill=black] (5.4,2.8) circle (2pt);
\draw [fill=white] (2.5,5.5) circle (2pt);
\draw [fill=black] (1.5,3.5) circle (2pt);
\draw [fill=black] (4,4.5) circle (2pt);
\draw [fill=black] (-4,4) circle (2pt);
\draw [fill=white] (-4.5,5.5) circle (2pt);
\draw [fill=white] (-3.5,5.5) circle (2pt);
\draw [fill=black] (-5.5,4) circle (2pt);
\draw [fill=black] (-5.5,3.5) circle (2pt);
\draw [fill=black] (-6,3.5) circle (2pt);
\draw [fill=black] (-5.5,2.5) circle (2pt);
\draw [fill=white] (-5.5,1.5) circle (2pt);
\draw [fill=white] (-5,1.5) circle (2pt);
\draw [fill=black] (-5,3) circle (2pt);




\draw [line width=0.8pt] (-7, 5.7)-- (-3.2, 5.7);
\draw [line width=0.8pt] (-7, 5.7)-- (-7, 1.3) ;
\draw [line width=0.8pt] (-3.2, 5.7)-- (-3.2, 1.3);
\draw [line width=0.8pt] (-7, 1.3)-- (-3.2, 1.3);
\end{scriptsize}
\end{tikzpicture}
\caption{Illustration of the proof of Lemma~\ref{lem: sprinkling event}. Here, the set $V_{q'} = V(\cC_l)\cup V(\cC_s)$ is drawn in red. Vertices landing there are black, and the vertices outside are white. 
The set $S$ is depicted in blue, and edges in the figure belong to $E_{q'}$. The large subset $S'$ of $S$ is at distance 1 from $\cC_l$ in $E_{q'}$.
Note that the set $A$ might `miss' some of the vertices in $S'\cap V(\cC_l)$, since some components in $\cC_l$ may be only partially in $S'$.}
\label{fig:5.4}
\end{figure}

\begin{proof}[Proof of Lemma \ref{lem: sprinkling event}]
Fix $p',q'$ such that $(1-p')(1-\eps)=1-p$ and $(1-q')(1-\eps)=1-q$, and $s\in \{d^2,\ldots, 3|V(Q^d_{p,q})|/4\}$. Consider independent random sets $V_{p'}$, $E_{q'}$, $V_{\eps}$ and $E_{\eps}$ and note that $Q^d_{p,q}$ and $(V_{p'}\cup V_{\eps},E_{q'}\cup E_{\eps})$ share the same distribution.
Denote by $X\subseteq V(Q^d)$ the set of vertices $v\in V(Q^d)$ with $|N_{E_{q'}}(v)\cap V_{p'}|\le \alpha d/8$. 

Let $\cC_s$ (respectively $\cC_l$) be the set of components of $(V_{p'},E_{q'})$ of size less than (respectively at least) $d^2$. 
Observe that every $C\in \cC_s$ has at least half its vertices in $X$: indeed, otherwise we would have a component $C$ of size at most $d^2$ such that $e(V(C),V(Q^d)\setminus V(C))\le d|V(C)|-(|V(C)|/2)\cdot (\alpha d/8)/2$
edges, contradicting Lemma \ref{l: Harper edge}.

Fix a disconnected set $S$ in $Q^d_{p,q}$ of size $s$ and with at most $\beta s/\sqrt{d}$ vertices with degree at most $\alpha d/4$ in $Q^d_{p,q}$.
Then, $S\cap V_{p'}$ is also disconnected in $(V_{p'},E_{q'})$, and therefore every component in $(V_{p'},E_{q'})$ is either disjoint from or contained in $S$. 

Denote by $X'\subseteq X$ the set of vertices in $X$ with degree at least $\alpha d/4$ in $Q^d_{p,q}$. 
By the definition of $S$ and our observation on $\cC_s$ from the second paragraph of the proof, we have
\begin{equation}\label{eq:snd}
|S\cap X|\le |S\cap X'|+\beta s/\sqrt{d}\qquad \text{and}\qquad |S\cap V(\cC_s)|\le 2|S\cap X|.
\end{equation}
Let $Y$ be the set of vertices $v$ in $V(Q^d)\setminus(X\cup V_{p'})$ with $N_{E_{q'}}(v)\cap V(\cC_l)=\varnothing$. 
Because $Y$ is disjoint from $X$, each vertex $v\in Y$ satisfies $|N_{E_{q'}}(v)\cap V(\cC_s)|\geq \alpha d/8$. Since $S$ is disconnected in $Q^d_{p,q}$, we have that every $v\in S\cap Y$ has at least $\alpha d/8$ neighbours in $Q^d_{p,q}$ which are all in $S\cap V(\cC_s)$. By counting the edges between $S\cap Y$ and $S\cap V(\cC_s)$ in two ways, we get $|S\cap Y|\leq 8/\alpha\cdot |S\cap V(\cC_s)|$. 

Note that every vertex in $X'$ had degree at most $\alpha d/8$
before the sprinkling of $V_{\eps}$ and $E_{\eps}$, but has degree
at least $\alpha d/4$ in $Q^d_{p,q}$. Hence each such vertex
gained at least $\alpha d/8$ additional edges from the
sprinkling. Therefore, by Lemma~\ref{lem: decent vertices}, with
probability at least $1-2^{-\eps^5 s/\sqrt{d}}$, we have
$|X'| \le \eps^5 s/\sqrt d$. By~\eqref{eq:snd}, on this event, we have
\begin{equation}\label{eq:chain}
|S\cap X|+|S\cap V(\cC_s)|+|S\cap Y|\le  |S\cap X|+(1+8/\alpha)|S\cap V(\cC_s)| \le |S\cap X| \cdot 19/\alpha
    \le (2\eps^5 s/\sqrt{d})\cdot 19/\alpha,    
\end{equation}
where the last inequality holds for $\beta\le \eps^5$.
Next, set $S'=S\setminus (X\cup V(\cC_s)\cup Y)$ and note that every $v\in S'$ is either in $V(\cC_l)$ or satisfies $N_{E_{q'}}(v)\cap V(\cC_l)\neq\varnothing$. 
Then, by assuming~\eqref{eq:chain},
$|S'|\ge (1-38\eps^5/(\alpha \sqrt{d}))s$ and, since $S$ is disconnected in $Q^d_{p,q}$, $N_{Q^d_{p,q}}(S')\subseteq S\setminus S'$. 

To conclude, we will show that such $S'$ exists with probability at most $2^{-\eps^5 s/\sqrt{d}}$ and therefore the probability there exists a set $S$ violating the statement is at most $2^{-\eps^5 s/\sqrt{d}}+2^{-\eps^5 s/\sqrt{d}}\le 2^{-\eps^5 s/(2\sqrt{d})}$, concluding the proof. 

To that end, set $\cC'\coloneqq \{C\in \cC_{\ell}\colon V(C)\cap S'\neq\varnothing\}$. Let $A=A(\cC')\subseteq V(Q^d)$ be the set of all $v\in V(Q^d)$ for which there exists $u\in V(\cC')$ such that $uv\in E_{q'}$. Note that $A$ contains $V(\cC')$ but also potentially some vertices outside $V_{p'}$.
Observe that $S'\subseteq A$ and hence $|A|\ge |S'|\ge s/2$. 
Further, since $S$ has an empty neighbourhood in $Q^d_{p,q}$, every $C\in \cC'$ is entirely contained in $S$. Further, every $v\in A\cap V(Q^d_{p,q})$ is is either a vertex of such $C$, and thus in $S$, or is an $E_{q'}$-neighbour (and thus a neighbour in $Q^d_{p,q}$) of some vertex in such a $C$. We thus conclude that $A$ is disjoint from $V(Q^d_{p,q})\setminus S$. Now, by Lemma \ref{l: chernoff}, with probability at least $1-\e^{-p2^d/12}$, we have that $|V(Q^d_{p,q})|\ge p2^d/2$ and, since $s\le 3|V(Q^d_{p,q})|/4$, we also have \[|A|\le 2^d-|V(Q^d_{p,q})\setminus S|\le 2^d-(|V(Q^d_{p,q})|-s)\le 2^d-|V(Q^d_{p,q})|/4\le 2^d(1-p/8).\]
By choosing $\eps\le p/8$ and recalling that $|A|\ge s/2$, we apply Lemma \ref{lem: expanding matching} for $A=A(\cC')$, $U=V_{p'}$ and $F=E_{q'}$ to obtain that, with probability at least $1-2^{-\eps^4s/(980\sqrt{d})}$, in $(V_\eps,E_\eps)\subseteq Q^d_{p,q}$, there is a matching $M$ with $e(M)\ge \eps^4s/(160\sqrt{d})$ leaving $A$. 
Moreover, by~\eqref{eq:chain}, with probability at least $1-2^{-\eps^5 s/\sqrt{d}}$, at most $38\eps^5 s/(\alpha\sqrt{d})$ edges of this matching have an endpoint in $S\setminus S'$. For sufficiently small $\eps$, this leads to contradiction since,
\[|N_{Q^d_{p,q}}(S')|-|S\setminus S'|\ge e(M)-2|S\setminus S'|\ge \eps^4s/(160\sqrt{d}) - 38\eps^5s/(\alpha \sqrt{d}) > 0.\qedhere\]
\end{proof}

We are now ready to prove the key proposition of this section.
\begin{proof}[Proof of Proposition \ref{prop: expansion of high degree}]
Consider an integer $s\le d^2$.
Then, by Lemma \ref{l: Harper edge}, for every set $S$ with $|S|=s$ satisfying the assumptions of the proposition, we have that $e(S)\le s\log_2s=2s\log_2d$. 
On the other hand, we have that $e_{Q^d_{p,q}}(S,V(Q^d_{p,q}))\ge (\alpha d/2)\cdot s$, and therefore 
    \begin{align*}
        e_{Q^d_{p,q}}(S,V(Q^d_{p,q})\setminus S)\ge (\alpha d/2)\cdot s-2e(S)\ge (\alpha d/2-4\log_2d)s\ge \alpha ds/3.
    \end{align*}
    Since $Q^d$ is $d$-regular, we obtain that $|N_{Q^d_{p,q}}(S)|\ge \alpha s/3>\gamma s/\sqrt{d}$.

    Next, assume $s>d^2$ and subsample each vertex of $Q^d_{p,q}$ independently and with probability $1/2$, thus forming a copy of $Q^d_{p/2,q}\subseteq Q^d_{p,q}$. We define the events
    \begin{itemize}
        \item $\mathcal{A}_s$ that there is a set $S$ of size $s$ such that $N_{Q^d_{p,q}}(S)<\gamma s/\sqrt{d}$ and, for every $v\in S$, the degree of $v$ in $Q^d_{p,q}$ is at least $\alpha d$, and
        \item $\cB_s$ that there is a set $S$ of size $s$ disconnected in $Q^d_{p/2,q}$, and at least $s-\frac{4\gamma s}{\alpha d^{1/2}}$ vertices in $S$ have degree at least $\alpha d/7$ in $Q^d_{p/2,q}$.
    \end{itemize}
    Now, condition on $\mathcal{A}_s$, reveal $Q^d_{p,q}$ and consider some fixed $S$ witnessing $\mathcal{A}_s$. 
    Then, 
    \begin{equation}\label{eq:cn1}
    \text{with probability at least $2^{-\gamma s/\sqrt{d}}$, $S$ is disconnected in $Q^d_{p/2,q}$,}
    \end{equation}
    as it suffices that all vertices in $N_{Q^d_{p,q}}(S)$ remain outside $Q^d_{p/2,q}$. 
    Let $S'\subseteq S$ be the set of all vertices $v\in S$ with $e_{Q^d_{p,q}}(v,V(Q^d_{p,q})\setminus S)\ge \alpha d/3$. 
    Since $S$ witnesses $\mathcal{A}_s$, it satisfies $N_{Q^d_{p,q}}(S)<\frac{\gamma s}{\sqrt{d}}$, implying that $|S'|\le \frac{(\gamma s/\sqrt{d})\cdot d}{\alpha d/3} = \frac{3\gamma s}{\alpha \sqrt{d}}$. 
    Denote by $X$ the random variable counting the vertices $v\in S\setminus S'$ satisfying $d_{Q^d_{p/2,q}}(v,S)<\alpha d/7$.
    For every $v\in S\setminus S'$, by Lemma~\ref{l: chernoff}, the probability that $d_{Q^d_{p/2,q}}(v,S)<\alpha d/7$ is at most $\mathbb P[\mathrm{Bin}(2\alpha d/3,1/2)<\alpha d/7]\le 2^{-\alpha d/900}$, so $\mathbb E[X]\le s 2^{-\alpha d/900}$ and, by Markov's inequality,
    \begin{equation}\label{eq:azumaeq}
    \begin{split}
    \mathbb P\bigg[X\ge \frac{\gamma s}{\alpha \sqrt{d}}\bigg] = o(1).
    \end{split}
    \end{equation}
    Thus, the probability that there are at most $\frac{4\gamma s}{\alpha \sqrt{d}}$ vertices $v\in S$ with $d_{Q^d_{p/2,q}}(v,S)<\alpha d/7$ is at least (say) $1/2$ and this event only depends on the states of the vertices of $S$ in $Q^d_{p/2,q}$.
    Putting~\eqref{eq:cn1} and~\eqref{eq:azumaeq} together, we have that 
    \begin{align}
        \mathbb P[\mathcal{B}_s\mid \mathcal{A}_s]\ge 2^{-\frac{\gamma s}{\sqrt{d}}-1}.\label{eq: event lower bound}
    \end{align}
    On the other hand, by applying Lemma \ref{lem: sprinkling event} (with $\alpha$ replaced by $\alpha/7$ and $\beta$ replaced by $4\gamma/\alpha$), there exists some $\eta'\coloneqq \eta'(p,q,\alpha)>0$ such that
    \begin{align}
        \mathbb P[\mathcal{B}_s]\le 2^{-\eta' s/\sqrt{d}}. \label{eq: event upper bound}
    \end{align}
    Putting inequalities \eqref{eq: event lower bound} and \eqref{eq: event upper bound} together, we obtain that $\mathbb P[\cA_s]\le \mathbb P[\cB_s]2^{\gamma s/\sqrt{d}+1}\le 2^{-\eta s/\sqrt{d}}$ for any $\gamma$ sufficiently small in terms as $\eta'$ and $\eta\in (0,\eta'-\gamma)$. As $s\ge d^2$, the last bound is at most $2^{-\eta d^{3/2}}$, as desired.
\end{proof}

\section{Coupling mixed percolation and long edges in the polytope}\label{section: family of cubes}
This section contains the main new structural ingredient of our work.
For an integer $k\in [d]$, recall that $Q^d(k)$ is the graph with vertex set $V(Q^d)$ and edges between every pair of vertices at Hamming distance $k$.
Although long edges in $P_p^d$ are highly dependent in general, we show that this dependence disappears on a carefully chosen family of cubes inside $Q^d(k)$. 
More precisely, for integers $k,m\ge 1$ with $km\le d$, we construct a family of subgraphs $\cQ_{k,m}$ of $Q^d(k)$, each isomorphic to $Q^m$. 
Then, for each $H\in\cQ_{k,m}$, Lemma~\ref{lem:reduction} identifies the induced subgraph obtained by $p$-vertex-percolating $H$ and restricting to the edges of $P_k$, with a mixed-percolated hypercube $Q^m_{p,q}$ for a constant $q=q(p,k)>0$. 
This exact coupling is what allows the local expansion theory from Proposition~\ref{prop: expansion of high degree} to be imported into the polytope setting. The rest of the section introduces the family $\cQ_{k,m}$ and the combinatorial facts needed to make the promised coupling precise.

Fix $k,m\ge 1$. While $Q^d(k)$ no longer has the same structure as the hypercube, it turns out that it contains a rich structure of subgraphs isomorphic to $Q^m$. 
We first create the necessary setup to exploit this structure. 
In the sequel, we associate the vertices of $Q^d$ with $\mathbb F_2^d$ in the usual way.
Note that $uv$ is an edge in $Q^d(k)$ if and only if $u$ and $v$ differ in exactly $k$ coordinates. Further, for every subset $e\subseteq [d]$, we denote by $\mathbf{1}_e\in\mathbb F_2^d$ the indicator vector of the set $e$.

First, we define a family of $m$ possible ``directions'' of length $k$. To that end, let $\cD_{k,m}$ be the set of all (unordered) sets $\{e_1,\ldots,e_m\}$ of size $m$ such that $e_1,\ldots,e_m$ are disjoint subsets of $[d]$ of size $k$. 
For every $D=\{e_1,\ldots, e_m\}\in \cD_{k,m}$ and $v\in V(Q^d)$, we define the hypercube $Q(D,v)$, which contains all vertices of the form $v+\sum_{i\in I}\mathbf{1}_{e_i}$ for every $I\subseteq [m]$, and for every $u,w\in V(Q(D,v))$, we have that $uw$ is an edge if and only if $u+w=e_i$ for some $i\in [m]$. We say that $e_1,\ldots,e_m$ are the directions of $Q(D,v)$. Observe that $Q(D,v)$ is a subgraph of $Q^d(k)$ isomorphic to $Q^m$. Additionally, $Q(D,v)$ does not depend on the choice of $v$ in the sense of the following lemma.
 \begin{lemma}\label{lem: choice of v}
     Let $D\in \cD_{k,m}$ and $v\in V(Q^d)$. Then, for every $u\in V(Q(D,v))$, it holds that $Q(D,v)=Q(D,u)$.
 \end{lemma}
 \begin{proof}
     Since $|V(Q(D,v))|=|V(Q(D,u))|=2^m$, it suffices to show $V(Q(D,u))\subseteq V(Q(D,v))$. Let $D=\{e_1,\ldots,e_m\}$. Since $u\in V(Q(D,v))$, there exists $I\subseteq [m]$ such that $u=v+\sum_{i\in I} \mathbf{1}_{e_i}$. Also, for every $u'\in V(Q(D,u))$, there exists $I'\subseteq [m]$ such that $u'= u+\sum_{i\in I'} \mathbf{1}_{e_i}$. But then, $u' = v+\sum_{i\in I} \mathbf{1}_{e_i} + \sum_{i\in I'} \mathbf{1}_{e_i} = v+\sum_{i\in J} \mathbf{1}_{e_i}$, where $J$ is the symmetric difference of $I$ and $I'$ (and thus $J\subseteq [m]$). Consequently, $u'\in V(Q(D,v))$.
 \end{proof}
 \begin{remark}\label{rem: partition}
 Lemma~\ref{lem: choice of v} has the following implication. Fix some $D\in\cD_{k,m}$. Then, for $u,v\in V(Q^d)$, we have that $Q(D,u)$ and $Q(D,v)$ are either the same or vertex-disjoint. As each $Q(D,v)$ contains exactly $2^m$ vertices, it follows that there are $2^{d-m}$ different hypercubes whose directions are given by $D$ and that these hypercubes partition $V(Q^d)$. 
 \end{remark}
 
Let $\cQ_{k,m}$ be the family of hypercubes given by $Q(D,v)$ varying over all possible choices of $D\in \cD_{k,m}$ and $v\in V(Q^d)$. Note that, as described above, $\cQ_{k,m}$ contains $2^{d-m}$ different hypercubes for every choice of $D\in\cD_{k,m}$, so that $|\cQ_{k,m}|=2^{d-m}|\cD_{k,m}|$. This observation allows us to get the following bound on the size of $\cQ_{k,m}$.
\begin{lemma}\label{lem: size of Qkm}
For all integers $k,m\ge 1$ with $km\le d$, $|\cQ_{k,m}|\leq e^{2d\log d}$.
\end{lemma}
\begin{proof}
Since every element of $[d]$ is in at most one of $e_1,\ldots, e_m$, the number of choices of $e_1,\ldots, e_m$ is at most $(m+1)^d\le e^{d\log(m+1)}$ (where every element is assigned to $e_1,\ldots,e_m$ or to no set at all).
Since $|\cQ_{k,m}|=2^{d-m}|\cD_{k,m}|$, we get $|\cQ_{k,m}|\leq 2^{d-m}\cdot e^{d\log (m+1)}\leq e^{2d\log d}$. 
\end{proof}

The next lemma is the main conclusion of this section. 
It shows that the interaction of the hypercubes in $\cQ_{k,m}$ with $P^d_p$ follows that of mixed percolation on the hypercube, paving the way to utilise Proposition \ref{prop: expansion of high degree} in later arguments. For graphs $G_1\subseteq G_2$ and an edge set $F\subseteq E(G_2)$, recall that $G_1[F]$ is the spanning subgraph of $G_1$ with edge set $E(G_1)\cap F$.

\begin{lemma}\label{lem:reduction}
For every $H\in \cQ_{k,m}$ with $H=(U,F)$, denote $H_p=(U_p,F)$.
Then, there is $q\geq (1-p)^{(2^k-2)}$ such that $H_p[P_k]$ has the same distribution as $Q^m_{p,q}$.
\end{lemma}
\begin{proof}
The vertices of $H_p$ are clearly present independently with probability $p$.
We show that, conditionally on its vertex set, the edges of $H_p[P_k]$ are present independently and with certain probability $q$.
To this end, recall that, by Remark~\ref{remark: edge via cube} the presence of an edge $uv\in Q^d(k)$ in $P^d_p$ depends only on the states of the vertices in $Q^d[u,v]$.
Hence, it suffices to show that the sets $\{V(Q^d[u,v])\setminus\{u,v\}: u,v\in V(H)\}$ are all disjoint from each other and $V(H)$: given this fact, each of them satisfies $V(Q^d_p[u,v])=\{u,v\}$ with probability $(1-p)^{(2^k-2)}$, which then justifies the bound on $q$ and suffices to conclude.


Denote by $D=\{e_1,\ldots,e_m\}\in \cD_{k,m}$ be the directions of $H$ so that $H=Q(D,u)$.
Also, fix two edges $uv,xy\in E(H)$ with $u+v=\mathbf{1}_{e'}$ and $x+y=\mathbf{1}_{e''}$. 
First, suppose $e'\neq e''$. By the definition of $H$, there exists $I\subseteq [m]$ such that $x=u+\sum_{i\in I} \mathbf{1}_{e_i}$ and $y=u+\mathbf{1}_{e''}+\sum_{i\in I} \mathbf{1}_{e_i}$. Since the $e_i$ are disjoint subsets of $[d]$, the restriction of one of $x$ or $y$ to the coordinates of $e''$ equals to that of $u$ and $v$; assume without loss of generality that it is $x$. But then, the restriction of any vertex in $Q^d[x,y]$ to $e''$, besides $x$, is different from that of $u$ and of $v$ and hence also different from all the vertices in $Q^d[u,v]$. It follows that $Q^d[u,v]$ and $Q^d[x,y]$ do not overlap besides potentially on $x$. 
If, on the other hand, $e'=e''$, then all of $u,v,x,y$ are distinct. It follows that $u$ and $x$ differ on $[d]\setminus e'$. But all the vertices in $Q^d[u,v]$ agree with $u$ on $[d]\setminus e'$ and all the vertices in $Q^d[x,y]$ agree with $x$ on $[d]\setminus e'$, implying that the cubes $Q^d[u,v]$ and $Q^d[x,y]$ are disjoint.
\end{proof}

We also remark that the edges in $Q^d(k)$ do \emph{not} percolate independently in general; Lemma~\ref{lem:reduction} implies that joint independence holds for the sets of edges of each cube in $\cQ_{k,m}$, and Proposition~\ref{prop: expansion of high degree} will be applied exclusively in this restricted setting. 

\section{\texorpdfstring{Expansion at distance $k$}{Expansion at distance k}}\label{sec: long distances}
This section extends the local expansion available inside the cubes $H\in\cQ_{k,m}$ to global expansion in the distance-$k$ graph. 
We work with a graph $Q_k\subseteq Q^d(k)$ 
which satisfies three important properties: 
first, $Q_k$ contains a significant proportion of the edges in $Q^d(k)$, second, after percolation, all cubes in the collection are ``typical'' in terms of their size and edge-expansion, and third, any intersections of cubes (smaller cubes themselves) also have ``typical'' size. 
Here, the second property ensures proper local expansion, while the third property is used to derive that large sets $S$ do not almost-fully intersect or remain almost-fully disjoint from almost-all cubes in the collection, thus allowing the promised local-to-global boosting.
Under these assumptions, the key lemma of this section (Lemma~\ref{lem: bootstrapped expansion}) states that every set of vertices where each vertex has degree at least $\varepsilon \tbinom{d}{k}$, edge-expands by a factor of $\tbinom{d}{k}/d^3$. 
The proof is an averaging argument over $\cQ_{k,m}$: Lemmas~\ref{lem:diameter} and~\ref{lem:cont} force many cubes into an intermediate regime where local expansion applies (thus contributing non-trivially to the global expansion in $Q_k$), while Lemma~\ref{lem: vertex is good} shows that a high-degree vertex is bad in only a negligible fraction of the cubes containing it. A final multiplicity count turns these local boundary contributions into a global lower bound.

In this section, we work with an odd integer $k\ge 3$ and $m\coloneqq  d/(3k^2)$. 
Fix $p,\varepsilon,\gamma\in(0,1)$, $Q_k\subseteq Q^d(k)$ and, for every $H\in \cQ_{k,m}$, write $H\cap  Q_k$ for the graph with vertex set $V(H)\cap V(Q_k)$ and edge set $E(H)\cap E(Q_k)$.
Throughout this section, we assume that $Q_k$ satisfies the following three conditions.
\begin{enumerate}[label=(\roman*), ref=(\roman*)]
    \item\label{itm: p} It holds that $|V(Q_k)|/2^d=p\pm o(1)$; 
    \item\label{itm: expansion} For every $H\in\cQ_{k,m}$ 
    and $S\subseteq V(H\cap  Q_k)$ with $|S|\leq 3|V(H\cap  Q_k)|/4$ and such that every vertex in $S$ has degree at least $\varepsilon^2d/(8k^2)$ in $H\cap  Q_k$, we have that $|N_{H\cap  Q_k}(S)|\ge \gamma|S|/\sqrt d$;
    \item\label{itm: cubes hit equally} For every $H\in\cQ_{k,m}$, it holds that $|V(H\cap Q_k)|= (1\pm \varepsilon)p2^m$ and, for every $H,H'\in\cQ_{k,m}$ with $|V(H)\cap V(H')|=2^{m-1}$, it holds that $|V(H)\cap V(H')\cap V(Q_k)|= (1/2\pm\varepsilon)p2^{m}$.
\end{enumerate}
Before continuing, we remark that $P_k$ (which, recall, is the subset of edges in $P^d_p$ with endpoints at Hamming distance $k$)
is shown to satisfy each of the above conditions in Section~\ref{sec: proof of Theorems}.

The following lemma is the main result of this section.

\begin{lemma}\label{lem: bootstrapped expansion}
Fix an odd integer $k$, constants $p,\varepsilon,\gamma\in(0,1)$ and $Q_k\subseteq Q^d(k)$ satisfying \emph{\ref{itm: p}}, \emph{\ref{itm: expansion}} and \emph{\ref{itm: cubes hit equally}}. Then, for every set $S\subseteq V(Q_k)$ of size $|S|\leq 3/5\cdot|V(Q_k)|$ such that every vertex in $S$ has degree at least $\varepsilon \tbinom{d}{k}$ in $Q_k$, we have that $e_{Q_k}(S,V(Q_k)\setminus S)\geq \tbinom{d}{k}\cdot|S|/d^3$.
\end{lemma}


To prepare the setup for the proof of Lemma~\ref{lem: bootstrapped expansion}, we first introduce an auxiliary graph $G_{\square}$ with vertex set $\cQ_{k,m}$ where $H,H'\in\cQ_{k,m}$ form an edge if $|V(H)\cap V(H')|=2^{m-1}$, 
and prove in Lemma~\ref{lem:diameter} that it has bounded diameter. This will allow us to transition between cubes containing many vertices of $S$ with cubes containing few such vertices via short paths. 
We then prove Lemma~\ref{lem:cont} showing that any path in $G_{\square}$ from a cube which heavily intersects $S$ to a cube which does not must pass through a cube which ``moderately'' intersects $S$. 
Finally, in Lemma~\ref{lem: vertex is good}, we show that a vertex of large degree in $Q_k$ fails to have large degree inside only a tiny fraction of the cubes containing it.

Together, these three ingredients imply that many cube-vertex incidences of $S$ occur inside cubes where the expansion assumption~\ref{itm: expansion} applies. 
Summing the resulting local boundary contributions and dividing by the number of times each edge is counted then gives the desired global lower bound on $e_{Q_k}(S,V(Q_k)\setminus S)$.


\begin{lemma}\label{lem:diameter}
$G_{\square}$ is a vertex-transitive graph and the diameter of $G_{\square}$ is at most $4kd$. 
\end{lemma}
\begin{proof}
The vertex-transitivity of $G_\square$ follows from the symmetry of $Q^d(k)$ under permutation of the coordinates and translations. We turn to show that the diameter of $G_{\square}$ is at most $4kd$. 

Fix $H,H'\in\cQ_{k,m}$. We construct a sequence $H=H_0,\ldots,H_{4kd}=H'\in\cQ_{k,m}$ such that, for each $i\in [0,4kd-1]$, either $H_i=H_{i+1}$ or $|V(H_i)\cap V(H_{i+1})|=2^{m-1}$. 
Note that, for all $D,D'\in\cD_{k,m}$ such that $|D\cap D'|=m-1$ and for all $w\in V(Q^d)$, $Q(D,w)$ and $Q(D',w)$ intersect in $2^{m-1}$ vertices. 
Thus, it is enough to show that, for every $i\in [0,2d-1]$, $H_i=Q(D,w)$ and $H_{i+1}=Q(D',w)$ for some $D,D',w$ as above.

Before constructing this sequence, we show that every two neighbours $u$ and $v$ in $Q^d$ are connected by a path of length 3 in $Q^d(k)$: indeed, if $uv$ is an edge in direction $i$, since $k+1$ is even, we have $u=v+\mathbf{1}_{e}+\mathbf{1}_{f}+\mathbf{1}_{g}$ for some triplet of $k$-sets where $i\notin g$, $|e\cap f|=(k-1)/2$ and the symmetric difference of $e$ and $f$ is $g\cup \{i\}$.
Therefore, for any $u\in V(H)$ and $v\in V(H')$, there exists a path $u=u_0,u_1,\ldots, u_r=v$ in $Q^d(k)$ with $r\le 3d$. 

First, we construct a subsequence $H=H_0,\ldots, H_{(k+1)r}\in \cQ_{k,m}$ with $v\in H_{(k+1)r}$ and, for all $i\in [0,(k+1)r-1]$, $H_i=Q(D,u_{\lfloor i/(k+1)\rfloor})$ and $H_{i+1}=Q(D',u_{\lfloor i/(k+1)\rfloor})$ for some $D,D'\in\cD_{k,m}$ with $|D\cap D'|\ge m-1$. We argue by induction. Suppose that $H_0,\ldots,H_{(k+1)i}$ have been constructed for some $i\in \{0,\ldots,r-1\}$.
Let $D\in \cD_{k,m}$ be such that $H_{(k+1)i}=Q(D,u_i)$. Denote by $e\subseteq [d]$ the set of $k$ coordinates where $u_i$ and $u_{i+1}$ differ. Let $e_1,\ldots,e_s\subseteq D$ be the sets in $D$ intersecting $e$ and note that, since $e$ has size $k$, $s\le k$. We define $H_{(k+1)i+1},\ldots,H_{(k+1)i+s}$ with $H_{(k+1)i+j}=Q(D_j,u_i)$, where $D_0 = D$ and $D_j$ is obtained from $D_{j-1}$ by replacing $e_j$ with arbitrary set $e'_j$ within $[d]\setminus (e\cup\bigcup_{e'\in D_{j-1}} e')$ of size $k$. Note that there are always enough available coordinates for $e_j$ since $k(m+1)\le d/(2k)$. 
Finally, fix any $e'\in D_{s}$ and let $D'$ be the set obtained from $D_s$ by replacing $e'$ with $e$. 
Setting $H_{(k+1)i+s+1}=\ldots=H_{(k+1)(i+1)}$ to be equal to $Q(D',u_i)$ which contains $u_i+e=u_{i+1}$ concludes the induction. 

Next, we extend the sequence by constructing $H_{(k+1)r+1},\ldots,H_{(k+1)r+2m}=H'$ such that, for every integer $i\in [(k+1)r,(k+1)r+2m-1]$, $H_i=Q(D,v)$ and $H_{i+1}=Q(D',v)$ for some $D,D'\in\cD_{k,m}$ with $|D\cap D'|\ge m-1$. 
Indeed, note that $H_{(k+1)r}=Q(D_{(k+1)r},v)$ and $H'=Q(D_{(k+1)r+2m},v)$ for some $D_{(k+1)r}$ and $D_{(k+1)r+2m}$ which differ in at most $2m$ sets. 
Denote by $E_1$ (respectively $E_2$) the union of the sets in $D_{(k+1)r}$ (respectively $D_{(k+1)r+2m}$), and
set $E=[d]\setminus (E_1\cup E_2)$. Since $|E_1|=|E_2|=km$ and $m=d/(3k^2)$, we have $|E|\ge km$. 
Let $D=\{e_1,\ldots, e_m\}$ be an arbitrary collection of $m$-many disjoint sets of size $k$ in $E$. For $i\in [m]$, we obtain $D_{(k+1)r+i}$ from $D_{(k+1)r+i-1}$ by replacing some set of $D_{(k+1)r}$ by some set of $D$. Then, for $i\in [(k+1)r+m,(k+1)r+2m]$, let $D_{(k+1)r+i}$ be obtained from $D_{(k+1)r+i-1}$ by replacing some set of $D$ by some set of $D_{(k+1)r+2m}$. 
Finally, for every $i\in [2m]$, set $H_{(k+1)r+i}=Q(D_{(k+1)r+i},v)$. 

The lemma follows by observing that $(k+1)r+2m\leq (k+1)(3d)+2d/(3k^2)\leq 4kd$ (recall $k\ge 3$) and defining $H_i=H'$ for every $i\in [(k+1)r+2m,4kd]$.
\end{proof}

We now show that any path in $G_{\square}$ from a cube which heavily intersects $S$ to a cube which does not, must pass through a cube which ``moderately'' intersects $S$. Formally, given a set $S\subseteq V(Q_k)$, we say that $H\in\cQ_{k,m}$ is \textit{$S$-dense} if it contains at least $2/3\cdot p2^m$ vertices of $S$. Further, we say that $H$ is \textit{$S$-moderate} if it contains between $1/7\cdot p2^m$ and $2/3\cdot p2^m$ vertices of $S$.

\begin{lemma}\label{lem:cont}
Fix $p\in(0,1)$ and $Q_k\subseteq Q^d(k)$ satisfying \emph{\ref{itm: cubes hit equally}}.
Consider a set $S\subseteq V(Q_k)$ and $H,H'\in\cQ_{k,m}$ such that $H$ is $S$-dense and $H'$ is not $S$-dense. 
Then, every path in $G_{\square}$ from $H$ to $H'$ contains an $S$-moderate vertex.
\end{lemma}
\begin{proof}
Fix a path from $H$ to $H'$, denote by $Q\in \cQ_{k,m}$ the last $S$-dense cube on it, and let $Q'\in\cQ_{k,m}$ be the next cube; note that such $Q,Q'$ exists since $H'$ is not $S$-dense. By \ref{itm: cubes hit equally}, we have that $|V(Q)\cap V(Q_k)|\leq (1+\varepsilon)p2^m$ and $|V(Q)\cap V(Q')\cap V(Q_k)|\geq (1/2-\varepsilon)p2^{m}$. Since $Q$ is $S$-dense, we have that $|(V(Q)\cap V(Q_k))\setminus S|\leq (1/3+\varepsilon)p2^m$. Therefore, 
\begin{align*}
   |V(Q')\cap S|\geq |V(Q)\cap V(Q')\cap V(Q_k)\cap S|\geq (1/2-\eps)p2^m-(1/3+\eps)p2^m=(1/6-2\varepsilon)p2^m\geq 1/7\cdot p2^m. &\qedhere
\end{align*}
\end{proof}

We now show that a vertex of large degree in $Q_k$ fails to have large degree inside only a very small fraction of the cubes containing it. Set $N=|\cQ_{k,m}|$. 
\begin{lemma}\label{lem: vertex is good}
Fix an odd integer $k\ge 1$, $\eps=\eps(k)$ suitably small and a vertex $v\in V(Q_k)$ with degree at least $\varepsilon \tbinom{d}{k}$ in $Q_k$. Then, there are at most $N2^{m-d-\varepsilon^2 d/(48k^2)}$ hypercubes $H\in\cQ_{k,m}$ with $v\in V(H)$ such that $v$ has degree less than $\varepsilon^2 d/(8k^2)$ in $H\cap Q_k$.
\end{lemma}
\begin{proof}
    As every $H\in\cQ_{k,m}$ contains $2^m$ vertices of $Q^d$, we get by symmetry and by Remark~\ref{rem: partition} that every vertex of $Q^d$, in particular $v$, is in exactly $N2^{m-d}$ hypercubes in $\cQ_{k,m}$.
    
    Denote by $D_v\subseteq \binom{[d]}{k}$ the family of $k$-subsets $e$ such that $v(v+\mathbf{1}_{e})$ is an edge in $Q_k$. 
    By assumption, we have $|D_v|\geq \varepsilon \binom{d}{k}$. 
    By Remark~\ref{rem: partition}, for every family $D$ of disjoint subsets $e_1,\ldots,e_m$ of size $k$ of $[d]$, there is exactly one cube $H\in\cQ_{k,m}$ which contains $v$, namely $H=Q(D,v)$, and the degree of $v$ in $H$ is $|D\cap D_v|$. 
    Therefore, by the probabilistic method, the statement follows once we show that a uniformly random choice of $D\in \cD_{k,m}$ satisfies $|D\cap D_v|<\varepsilon^2 d/(8k^2)$ with probability at most $2^{-\varepsilon^2 d/(48k^2)}$. 
    
    Consider the following procedure to sample $D$ uniformly at random. We iteratively sample $e_1,\ldots,e_m$, where we select $e_i$ uniformly at random from $\binom{[d]\setminus(e_1\cup\cdots\cup e_{i-1})}{k}$. 
    Note that the resulting sets $e_1,\ldots,e_m$ are chosen uniformly at random under the condition of being disjoint so that $D=\{e_1,\ldots,e_m\}$ is sampled uniformly at random from $\cD_{k,m}$. 
    
    Set $E_i=e_1\cup\cdots\cup e_i$ and $i_{\max} = \eps d/(2k^2)$. 
    We bound the probability that the process at stage $i_{\max}$ does not yet witnesses the presence of $\eps^2 d/(8k^2)$ edges.
    First, for every $i\leq i_{\max}$, we have $|E_i|\leq \varepsilon d/(2k)$. Further, every $j\in[d]$ is in 
    $\binom{d-1}{k-1}$
    sets of size $k$. 
    Thus, for every $i\le i_{\max}$, the number of sets $e\in D_v$ which do not intersect any of the $e'\in E_i$ is at least $|D_v|-|E_i|\cdot \binom{d-1}{k-1}\ge \eps \tbinom{d}{k}/2$. 
    Since the number of ways to choose $e_{i+1}$ uniformly at random is at most 
    $\tbinom{d}{k}$, it follows that $e_{i+1}\in D_v$ with probability at least $\varepsilon/2$. Therefore, $|D\cap D_v|$ stochastically dominates $\mathrm{Bin}(i_{\max},\varepsilon/2)$. By Lemma \ref{l: chernoff}, we get that the probability there are less than $\varepsilon^2 d/(8k^2)$ sets of $D_v$ in $D$ is at most $e^{-\varepsilon^2d/(48k^2)}\leq 2^{-\varepsilon^2d/(48k^2)}$.
\end{proof}

With these tools at hand, we are now ready to prove Lemma~\ref{lem: bootstrapped expansion}. 
\begin{proof}[Proof of Lemma~\ref{lem: bootstrapped expansion}]
Fix any set $S\subseteq V(Q_k)$ with $|S|\leq 3/5\cdot |V(Q_k)|$ such that every vertex in $S$ has degree at least $\varepsilon \tbinom{d}{k}$ in $Q_k$. Let $\cQ_{dense}\subseteq\cQ_{k,m}$ be the set of hypercubes which are $S$-dense and set $\cQ_{sparse} = \cQ_{k,m}\setminus\cQ_{dense}$. 
Further, let $\cQ\subseteq \cQ_{sparse}$ be the family of hypercubes $H\in \cQ_{sparse}$ in which all but at most $(1/d^2)$-proportion of the vertices of $S\cap V(H)$ have degree at least $\eps^2d/(8k^2)$ in $H\cap Q_k$.

Further, fix $H\in \cQ$ and denote by $S'\subseteq S\cap V(H)$ the set of vertices with degree at least $\eps^2d/(8k^2)$ in $H\cap Q_k$. 
Then, by \ref{itm: expansion} applied to $S'$, $|N_{H\cap Q_k}(S')|\ge \gamma |S'|/\sqrt{d}\ge (1-1/d^2)\gamma |S\cap V(H)|/\sqrt d$. Furthermore, by using that $|(S\cap V(H))\setminus S'|\le |S\cap V(H)|/d^2$, we obtain
\begin{align*}
  e_{H\cap Q_k}(S,V(H\cap Q_k)\setminus S)
  \geq (1-1/d^2)\gamma|S\cap V(H)|/\sqrt d-d\cdot |S\cap V(H)|/d^2\geq \gamma|S\cap V(H)|/(2\sqrt d).  
\end{align*} 

Recall our aim to estimate $e_{Q_k}(S, V(Q_k)\setminus S)$ as well as the relations $N=|\cQ_{k,m}|$, $e(Q^d(k)) = \tbinom{d}{k} 2^{d-1}$ and, for every $H\in \cQ_{k,m}$, $V(H)=2^m$ and $e(H)=m2^{m-1}$.
Moreover, by symmetry and double-counting of the pairs $(H,e)$ where $H\in \cQ_{k,m}$ and $e\in E(H)$, every edge of $Q^d(k)$ appears in the same number of hypercubes in $\cQ_{k,m}$, namely $Nm2^{m-1}/(\binom{d}{k}2^{d-1})$. Simplifying and using that $m\le d$ gives,
\begin{align}\label{eq: edges leaving S}
    e_{Q_k}(S,V(Q_k)\setminus S)\geq \left(\sum_{H\in \cQ}\frac{\gamma|S\cap V(H)|}{2\sqrt d}\right)\bigg/\left(\frac{N d 2^{m-d}}{\binom{d}{k}}\right).
\end{align}

For brevity, define 
\[\Sigma\coloneqq\sum_{H\in \cQ_{k,m}}|S\cap V(H)|,\; \Sigma_d\coloneqq\sum_{H\in \cQ_{dense}}|S\cap V(H)|,\; \Sigma_s\coloneqq\sum_{H\in \cQ_{sparse}}|S\cap V(H)|,\; \Sigma_{\cQ}\coloneqq\sum_{H\in \cQ}|S\cap V(H)|.\]

Next, we work towards showing that $\Sigma_{\cQ}\ge |S|\cdot N2^{m-d}/(330kd)$. By symmetry, every vertex $v\in V(Q^d(k))$ appears in $N\cdot 2^{m-d}$ hypercubes in $\cQ_{k,m}$ and thus $\Sigma= |S|N2^{m-d}$. Before dealing with $\Sigma_\cQ$, we first show that $\Sigma_s\geq \Sigma/(320kd)$. 

Fix a vertex $H\in V(G_{\square})$ and a family $\cP_H$ of paths of length at most $4kd$ containing, for every $H'\in V(G_{\square})$ distinct from $H$, exactly one path from $H$ to $H'$. 
By Lemma~\ref{lem:diameter}, such a choice of $\cP_H$ exists. Recall that $G_\square$ is vertex-transitive and denote by $\mathrm{Aut}(G_{\square})$ the automorphism group of $G_{\square}$. 
Define the family of paths $\cP'=\bigcup_{\psi\in\mathrm{Aut}(G_{\square})}\psi(\cP_H)$ (where repeated paths appear with multiplicity). Since $G_{\square}$ is vertex-transitive, the number of automorphisms mapping $H$ to any given vertex $\Gamma\in V(G_{\square})$ is the same; hence, for each $\Gamma\in V(G_{\square})$, this number equals $|\mathrm{Aut}(G_{\square})|/N$. Therefore, for any distinct pair $\Gamma,\Gamma'\in V(G_{\square})$, the family $\cP'$ contains the same number $L=\frac{2|\mathrm{Aut}(G_{\square})|}{N}$ of paths with endpoints $\Gamma$ and $\Gamma'$. Likewise, every vertex appears on the same number $M$ of paths in $\cP'$. Let us count the number of pairs $(v,P)$ for $P\in \cP'$ and $v\in V(P)$.
On one hand, $G_\square$ has $N$ vertices, each of which appears by definition on $M$ paths, so there are $N\cdot M$ such pairs. 
On the other hand, there are $L$ paths in $\cP'$ between any distinct pair in $V(G_\square)$. As each of these paths has length at most $4kd$, the number of pairs $(v,P)$ is at most $4kd\cdot L\cdot \binom{N}{2}$. We obtain $M\cdot N\leq 4kd \cdot L\cdot \binom{N}{2}$ and rearranging yields 
\begin{align}\label{eq: path double count}
    M\leq 2kd\cdot L\cdot (N-1)\leq 2kd\cdot L\cdot N.
\end{align}

Denote by $X$ the number of $S$-dense $H\in\cQ_{k,m}$. Recall $\Sigma=|S|N2^{m-d}\leq (3/5+o(1))pN2^{m}$, where we use \ref{itm: p} and $|S|\leq 3/5\cdot |V(Q_k)|$. Using \ref{itm: cubes hit equally}, it also holds that $\Sigma\geq X\cdot (2/3+o(1))\cdot p2^m$. 
Putting these two inequalities together, we obtain $X\leq (3/5+o(1))(3/2+o(1))\cdot N\leq 10N/11$. 
Since for any pair of vertices $u\neq v$ of $G_{\square}$, $\cP'$ contains $L$ paths between $u$ and $v$, 
it follows that $\cP'$ contains at least $L\cdot X\cdot (N-X)\ge L\cdot X\cdot N/11$ paths for which one endpoint is $S$-dense and the other is not. By Lemma~\ref{lem:cont}, every such path contains an $S$-moderate cube. 
Finally, as each $S$-moderate cube is on $M$ paths, using \eqref{eq: path double count}, there are at least $LXN/(11M)\geq X/(22kd)$ many $S$-moderate cubes. 

Thus, for any constant $\eps>0$, $\Sigma_d\le X\cdot (1+\varepsilon)p2^m$ and $\Sigma_s\ge (X/(22kd))\cdot (1-\varepsilon)p2^m/7$.
Hence, if $\Sigma_d\geq \Sigma/2$, 
\begin{equation}\label{eq:bdsigma_s}
\Sigma_s\ge \frac{(1-\eps)Xp2^m}{154kd}\ge \frac{(1-\eps)\Sigma_d}{154(1+\eps)kd}\ge \frac{\Sigma}{320kd}= \frac{|S| N 2^{m-d}}{320kd}.
\end{equation}
On the other hand, if $\Sigma_d\le \Sigma/2$, then $\Sigma_s= \Sigma-\Sigma_d\ge \Sigma/2$, so that in either case $\Sigma_s$ satisfies the bound from~\eqref{eq:bdsigma_s}.

We are ready to show our initial goal, namely that $\Sigma_{\cQ}\geq |S| N2^{m-d}/(330kd)$. For $H\in\cQ_{sparse}$, denote by $v_<(H,S)$ the number of vertices in $V(H)\cap S$ with degree less than $\varepsilon^2d/(8k^2)$ in $H\cap Q_k$. 
Then,
\[\sum_{H\in\cQ_{sparse}\setminus\cQ}|S\cap V(H)| \leq d^2\cdot \sum_{H\in\cQ_{sparse}\setminus\cQ}v_<(H,S), \text{ and by Lemma~\ref{lem: vertex is good},} \sum_{H\in\cQ_{k,m}}v_<(H,S)\leq |S|\cdot N2^{m-d-\varepsilon^2d/(48k^2)}.\] 
Hence, 
\begin{align*}
\Sigma_{\cQ} = \Sigma_s - \sum_{H\in \cQ_{sparse}\setminus \cQ}|S\cap V(H)|\geq \frac{|S| N 2^{m-d}}{320kd}-d^2|S|\cdot N2^{m-d-\varepsilon^2d/(48k^2)}\geq \frac{|S| N 2^{m-d}}{330kd}.
\end{align*}
Inserting this inequality in \eqref{eq: edges leaving S}, we obtain
\begin{align*}
    e_{Q_k}(S,V(Q_k)\setminus S)\geq \left(\frac{\gamma |S| N 2^{m-d}}{660kd^{3/2}}\right)\bigg/\left(\frac{Nd2^{m-d}}{\tbinom{d}{k}}\right)\geq \frac{\gamma|S| \tbinom{d}{k}}{660k d^{5/2}}\geq \binom{d}{k} \frac{|S|}{d^3}.&\qedhere
\end{align*}
\end{proof}

\section{Deriving Theorems \ref{thm:main1} and \ref{thm:main2}}\label{sec: proof of Theorems}

We start this section with one last auxiliary lemma about the degrees in the polytope at distance $k$. 
Recall that $P_k$ denotes the subset of edges in $P^d_p$ with endpoints at Hamming distance $k$.

\begin{lemma}\label{lem:min degree}
Fix an integer $k\ge 2$ and $p,\alpha\in (0,1)$. Then, whp every vertex in $P^d_p$ is either incident to at least $(1-\alpha)d$ edges in $P_1$ or to at least $p(1-p)^{2^k-2} \alpha^k \tbinom{d}{k}/4$ edges in $P_k$.
\end{lemma}
\begin{proof}
Fix $q \coloneqq p(1-p)^{2^k-2}$ and, for every $v\in V(Q^d)$, denote by $A_v$ the event that $v\notin V(P^d_p)$ or $v$ has degree at least $(1-\alpha)d$ in $P_1$ or at least $q \alpha^k \tbinom{d}{k}/4$ in $P_k$. 
We show that, for every $v\in V(Q^d)$, $\mathbb P(A_v^c)=o(2^{-d})$; this suffices to conclude by a union bound. 
Fix a vertex $v$ and reveal the percolation on $v\cup N_{Q^d}(v)$. If $v\notin V(P^d_p)$ or at least $(1-\alpha)d$ vertices in $N(v)$ survive then $A_v$ holds, so suppose this is not the case. 
Denote by $U\subseteq N(v)$ the set of vertices which did not survive the percolation, and by $W$ be the set of vertices $w$ at distance $k$ from $v$ such that all paths from $v$ to $w$ in $Q^d[v,w]$ intersect $U$. 
In particular, $|W|=\binom{|U|}{k}$. For $w\in W$, if the only vertices in $Q^d[v,w]$ which survive percolation are $v$ and $w$, then by Remark~\ref{remark: edge via cube} we have $vw\in P_k$. Hence, $vw$ is present with probability at least $q$. 
Let $X$ denote the number of neighbours of $v$ in $P_k$. Then, $\mathbb E[X]\geq q|W| \geq q(\alpha^k/2)\tbinom{d}{k}$ since we assume that $|U|\ge \alpha d$ and $k$ is fixed. 
Finally, conditionally on the set $U$, we apply Lemma \ref{l: azuma} for $X$.
Observe that, for any vertex $u$ at distance $\ell\in [2,k]$ from $v$, if we change its status, we affect at most $\binom{d}{k-\ell}$ vertices $w\in W$ for which $u\in V(Q^d[v,w])$, and there are $\binom{d}{\ell}$ vertices at distance $\ell$ from $v$. Thus, by Lemma \ref{l: azuma},
\begin{align*}
    \mathbb{P}\left[X\le q|W|/2\right]\le \mathbb{P}\left[|X-\mathbb{E}[X]|\ge q|W|/2\right]
    &\le 2\exp\left(-\frac{(q|W|/2)^2}{2\sum_{\ell=2}^k\binom{d}{\ell}\binom{d}{k-\ell}^2+\frac{2q|W|}{6}\binom{d}{k-2}}\right)\\
    &\le \exp\left(-\frac{(q(\alpha^k/4)\tbinom{d}{k})^2}{3\binom{d}{2}\binom{d}{k-2}^2}\right)\le \exp\bigg(-\frac{q^2 \alpha^{2k} d^2}{100k^4}\bigg)=o(2^{-d}).&\qedhere
\end{align*}
\end{proof}

We begin with the proof of Theorem \ref{thm:main2}.
\begin{proof}[Proof of Theorem~\ref{thm:main2}]
Decompose $Q^d=Q^{d-b}\times Q^b$ as described in Section~\ref{subsec:renorm} and note that, for every integer $t\ge 0$,
\[1-(1-p)^{2^t}\le 1-(1-p)^{2^{t+1}}\le 2(1-(1-p)^{2^t}).\]
Thus, we choose $b$ so that $\rho \coloneqq 1-(1-p)^{2^b}\in[1/100,1/50]$ when $p<1/100$, and set $b=0$ when $p\in[1/100,1/2-\varepsilon]$ (in this case $\rho=p$). 
To set the argument up, we first establish the following three claims.
\begin{claim}\label{itm: small intersection}
Whp, for every $b$-dimensional cube $Q$ from the decomposition of $Q^d$, it holds that $|V(Q)\cap V(P^d_p)|\leq d$.
\end{claim}
\begin{proof}
     Fix $Q\in V(Q^{d-b})$ and note that $\mathbb P[|V(Q)\cap V(P^d_p)|>d]=\mathbb P[\mathrm{Bin}(2^b,p)>d]$. 
     If $p\geq 1/100$, then $b=0$ and the above probability is $0$. 
     Assume $p<1/100$. Then, we have $e^{-p{2^b}}\ge (1-p)^{2^b}\ge 49/50>e^{-1/10}$, and therefore $p\cdot 2^b< 1/10$. As a result, by using the inequality $\tbinom{a}{b}\le (ea/b)^b$ for integers $a\ge b\ge 1$, we have
    \begin{align*}
        \mathbb P[|V(Q)\cap V(P^d_p)|>d]\le \sum_{i=d}^{2^b}\binom{2^b}{i}p^i\le 2^b\left(\frac{ep2^b}{d}\right)^{d}<2^b\left(\frac{e}{10d}\right)^{d}.
    \end{align*}
    A union bound over the $2^{d-b}$ choices of $Q$ completes the proof.
\end{proof}

Recall the polytope $P^{d-b}_{\rho}$ obtained by taking the convex hull of the points in $Q^{d-b}$ after vertex-percolation with probability $\rho$: these correspond to cubes in the decomposition of $Q^d$ containing at least one vertex of $P^d_p$.

\begin{claim}\label{itm: enough cubes}
Whp, the number of $b$-dimensional cubes $Q$ from the decomposition of $Q^d$ for which $|V(Q)\cap V(P^d_p)|>0$ is at least $(9/10)\cdot p2^{d}$ or, in other words, $|V(P^{d-b}_{\rho})|\geq (9/10)\cdot p2^d$.
\end{claim}
\begin{proof}
If $p\geq 1/100$, then $b=0$ and $P^{d-b}_\rho=P^{d}_p$. 
Hence, the claim follows from Lemma~\ref{l: chernoff}. Suppose then that $p<1/100$ and fix some $Q\in Q^{d-b}$. By the inclusion-exclusion principle, we obtain
\[\mathbb P[|V(Q)\cap V(P^d_p)|=0]=\mathbb P[\mathrm{Bin}(2^b,p)=0]\leq 1-2^bp+\binom{2^b}{2}p^2\leq 1-(1-2^bp/2)2^bp.\]
Similar to the proof of Claim~\ref{itm: small intersection}, we have $2^bp<1/10$ and thus
\[\mathbb P[|V(Q)\cap V(P^d_p)|>0] \ge (1-2^bp/2)2^bp \geq (19/20)2^bp.\]

Consequently, the number of $Q\in V(Q^{d-b})$ for which $|V(Q)\cap V(P^d_p)|>0$ follows a binomial distribution with expectation at least $(19/20)2^bp\cdot 2^{d-b}$. Thus, by Lemma~\ref{l: chernoff}, whp there are at least $(9/10)\cdot p2^{d}$ such $Q\in Q^{d-b}$.
\end{proof}
\begin{claim}\label{itm: high enough degree} 
Whp, every cube $Q$ as above has degree at least $p(1-p)^{2^k-2} \alpha^k \tbinom{d}{k}/4$ in $P^{d-b}_{\rho}(k)$.
\end{claim}
\begin{proof}
    By Lemma~\ref{lem: degrees}, we get that whp every $Q\in V(Q^{d-b})$ has degree at most $(1-\alpha)d$ in $Q^{d-b}_{\rho}=P^{d-b}_{\rho}(1)$. The claim then follows by Lemma~\ref{lem:min degree}.
\end{proof}
We are ready to prove the theorem.
Consider $S\subseteq V(P^{d}_p)$ with $|S|\leq |V(P^d_p)|/2$, and let $S'\subseteq V(Q^{d-b})$ be the set of $b$-dimensional cubes $Q$ in the decomposition of $Q^d$ which intersect $S$. 
By Claim~\ref{itm: small intersection}, we have that $|S'|\geq |S|/d$. Furthermore, by Lemma~\ref{l: chernoff}, whp it holds that $|V(P^d_p)|\leq (1+o(1))p2^d$, so that $|S'|\leq |S|\leq (1/2 + o(1))p2^d$. 
By Claim~\ref{itm: enough cubes}, we then obtain that $|S'|\leq (1/2+o(1))(10/9) |V(P^{d-b}_\rho)|\le 3/5\cdot |V(P^{d-b}_\rho)|$.
   
Let $k$ be an arbitrarily large odd constant, and set $m=d/(3k^2)$. We aim to apply Lemma~\ref{lem: bootstrapped expansion} with $Q_k$ given by the the graph of $P^{d-b}_{\rho}(k)$ to show that $S'$ expands well in $P^{d-b}_{\rho}(k)\subseteq P^{d-b}_{\rho}$. 
To this end, we need to demonstrate that \ref{itm: p}, \ref{itm: expansion} and \ref{itm: cubes hit equally} hold with high probability. 
Point \ref{itm: p} follows directly from Lemma~\ref{l: chernoff} as mentioned earlier. 
Further, recall that, by Lemma~\ref{lem: size of Qkm}, we have that $|\cQ_{k,m}|\leq e^{2d\log d}$. Thus, \ref{itm: cubes hit equally} follows from Lemma~\ref{l: chernoff} together with the union bound over all pairs $H,H'\in\cQ_{k,m}$. It remains to show \ref{itm: expansion}. 
To do so, fix $H\in\cQ_{k,m}$ and note that, by Lemma~\ref{lem:reduction}, $H\cap P^{d-b}_{\rho}(k)$ is distributed as $Q^m_{\rho,q}$ with $q$ constant. 
By Proposition~\ref{prop: expansion of high degree} and the fact that $m=\Theta(d/k^2)$, we arrive at the conclusion that $H$ satisfies \ref{itm: expansion} with probability at least $1-2^{-\omega(d\log d)}$.

Finally, by Lemma~\ref{lem: bootstrapped expansion} (which we can apply to $S'$ by Claim~\ref{itm: high enough degree}), we obtain that $e_{Q_k}(S',V(Q_k)\setminus S')\geq \tbinom{d}{k}|S'|/d^3$.
As Lemma~\ref{lem:normalization} yields $e_{P^{d}_p}(S,V(P^{d}_p)\setminus S)\geq \tbinom{d}{k}|S'|/d^3$, the bound $|S'|/d\geq |S|$ concludes the proof.
\end{proof}
We now turn to the proof of Theorem \ref{thm:main1}.

\begin{proof}[Proof of Theorem~\ref{thm:main1}]
By Theorem~\ref{thm:main2}, we may (and do) assume that $p\geq 1/3$. 
Fix $b=100$, $\rho=  1-(1-p)^{2^b} > 1-e^{-4}$ and consider the decomposition $Q^d=Q^{d-b}\times Q^b$ as described in Section~\ref{subsec:renorm}. 
Fix also $S\subseteq V(P_p^d)$ with $|S|\leq |V(P_p^d)|/2$, and $\alpha\in (0,1)$. 
If $|S|\leq 2^{(1-\alpha^3/2)d}$, let $S'\subseteq V(Q^{d-b})$ be the set of cubes $Q\in V(Q^{d-b})$ which intersect $S$. 
Note that $|S|/2^b\leq|S'|\leq |S|\leq 2^{(1-\alpha^3/2)d}$. 
By Lemma~\ref{lem: Wojtek}, we get that $e_{P^{d-b}_\rho}(S',V(P^{d-b}_\rho)\setminus S')\geq (|S'|/8)\cdot (\alpha^3 d/2-b)$. 
By Lemma~\ref{lem:normalization}, it follows that $e_{P^{d}_p}(S,V(P^{d}_p)\setminus S)\geq (|S'|/8)\cdot \alpha^3 d/4\geq |S|\cdot (\alpha^3 d/2^{b+5})$.
    
Suppose that $|S|\geq 2^{(1-\alpha^3/2)d}$. Combining Lemma~\ref{lem: degrees} and Lemma~\ref{lem:min degree} similarly to the proof of Claim~\ref{itm: high enough degree}, whp all but $2^{(1-\alpha^3)d}$ vertices of $P_p^d$ have degree at least 
$d' := p(1-p)^{2^k-2} \alpha^k \tbinom{d}{k}/4$ in $P_p^d(k)$. 
Denote by $S''\subseteq S$ the subset of vertices with degree at least $d'$ in $P_p^d(k)$ and note that, by the above, we may (and do) assume $|S''|\geq (1-2^{-\alpha^3 d/2})|S|$. 
Furthermore, by the upper bound on $S$, we have $|S''|\leq 3/5\cdot |V(P^d_{p}(k))|$.
Recalling that assumptions \ref{itm: p}, \ref{itm: expansion} and \ref{itm: cubes hit equally} are satisfied whp (shown in the proof of Theorem~\ref{thm:main2}), we can apply Lemma~\ref{lem: bootstrapped expansion} to conclude that $e_{P^d_p(k)}(S'',V(P^d_p(k))\setminus S'')\geq \tbinom{d}{k}|S''|/d^3$. 
Since also the maximum degree of the graph of $P^d_p(k)$ is at most $\tbinom{d}{k}$ we have \[e_{P^d_p(k)}(S,V(P^d_p(k))\setminus S)\geq e_{P^d_p(k)}(S'',V(P^d_p(k))\setminus S'')-|S\setminus S''|\cdot \tbinom{d}{k}\ge \tbinom{d}{k} |S''|/(2d^3). \qedhere\]
\end{proof}

\section{Concluding remarks}
In this paper, we essentially resolve the problem of the edge-expansion of random $0/1$ polytopes, proving in a strong form the Mihail--Vazirani conjecture in this setting; 
in particular, we showed that, for a typical $0/1$ polytope, the Mihail--Vazirani conjecture holds in a strong sense. We further establish that the edge-expansion exhibits a phase transition at $p = 1/2$. 

Our results suggest several open questions. First, for $p > 1/2$, Theorem~\ref{thm:main1} leaves a gap between our
lower bound of $cd$ and the upper bound of $d$. It would be interesting to understand the dependency between $c$ and $p$. Second, Theorem~\ref{thm:main2} establishes super-polynomial expansion when $p < 1/2$, and minor modifications suffice to establish the optimal exponent of $d$ for constant $p$, but the precise growth rate as a function of both $p$ and $d$ remains open when $p=p(d)\to 0$ fast enough.
It is natural to believe that this growth rate is given by the minimum degree of the graph of $P^d_p$, as this is the case for many mean-field random models with good expansion properties.
Confirming this suspicion and further showing that the minimum and the average degree of the graph are typically within a polynomial factor of $d$ from each other would lead to significantly enhanced (polynomial) control on the mixing time of the lazy random walk on $P^d_p$.

\paragraph{Acknowledgements} MC was supported by SNSF Ambizione grant No.~216071. LL was supported by the Austrian Science Fund (FWF) grant No.~10.55776/ESP624. BS was supported by SNSF grant No.~200021–228014. 
For open access purposes, the authors have applied a CC BY public copyright license to any author-accepted manuscript version arising from this submission.
Part of this work was done while LL was visiting FIM at ETH Z\"urich; the author is grateful to the institution for its support and hospitality.

\bibliographystyle{alpha}
{\footnotesize{\bibliography{big}}}

@inproceedings{FKSS26,
  title={On the edge expansion of random polytopes},
  author={Ferber, A. and Krivelevich, M. and Sales, M. and Samotij, W.},
  booktitle={Proceedings of the 2026 Annual ACM-SIAM Symposium on Discrete Algorithms (SODA)},
  pages={3022--3035},
  year={2026},
  organization={SIAM}
}

@book{LP17,
  title={Markov chains and mixing times},
  author={Levin, D.~A. and Peres, Y.},
  volume={107},
  year={2017},
  publisher={American Mathematical Society}
}

@article{H64,
    author  = "Harper, L. H.",
    title   = "Optimal assigments of numbers to vertices",
    year    = "1964",
    journal = "SIAM J. Appl. Math.",
    volume  = "12",
    pages   = "131-135"
}

@article{L64,
    author  = "Lindsey, J. H.",
    title   = "Assigment of numbers to vertices",
    year    = "1964",
    journal = "Am. Math. Mon.",
    volume  = "71",
    pages   = "508-516"
}

@article{B67,
    author  = "Bernstein, A. J.",
    title   = "Maximally connected arrays on the $n$-cube",
    year    = "1967",
    journal = "SIAM J. Appl. Math.",
    volume  = "15",
    pages   = "1485-1489"
}

@article{H76,
    author  = "Hart, S.",
    title   = "A note on the edges of the $n$-cube",
    year    = "1976",
    journal = "Discrete Math.",
    volume  = "14",
    pages   = "157-163"
}

@article{H66,
  AUTHOR = {Harper, L. H.},
     TITLE = {Optimal numberings and isoperimetric problems on graphs},
   JOURNAL = {J. Comb. Theory},
  FJOURNAL = {Journal of Combinatorial Theory},
    VOLUME = {1},
      YEAR = {1966},
     PAGES = {385--393},
      ISSN = {0021-9800},
   MRCLASS = {05.40},
  MRNUMBER = {200192},
MRREVIEWER = {D. W. Walkup},
}

@Book{AS16,
 Author = {N. {Alon} and J. H. {Spencer}},
 Title = {{The probabilistic method}},
 ISBN = {978-1-119-06195-3; 978-1-119-06207-3},
 Pages = {384},
 Year = {2016},
 EDITION = {Fourth},
 Publisher = {Hoboken, NJ: John Wiley \& Sons},
 Language = {English},
 MSC2010 = {05-02 05D40},
 Zbl = {1333.05001}
}

@article{ALOV,
    AUTHOR = {Anari, N. and Liu, K. and Oveis Gharan, S. and
              Vinzant, C.},
     TITLE = {Log-concave polynomials {II}: {H}igh-dimensional walks and an
              {FPRAS} for counting bases of a matroid},
   JOURNAL = {Ann. of Math. (2)},
  FJOURNAL = {Annals of Mathematics. Second Series},
    VOLUME = {199},
      YEAR = {2024},
    NUMBER = {1},
     PAGES = {259--299},
      ISSN = {0003-486X,1939-8980},
   MRCLASS = {68Q87 (60J10 68W20 68W25)},
  MRNUMBER = {4681146},
MRREVIEWER = {Yuzhou\ Gu},
       DOI = {10.4007/annals.2024.199.1.4},
       URL = {https://doi.org/10.4007/annals.2024.199.1.4},
}

@misc{ADLZ25,
 author = {Anastos, M. and Diskin, S. and Lichev, L. and Zhukovskii, M.},
 title = {Diameter and mixing time of the giant component in the percolated hypercube},
 year = {},
 note = {arXiv:2510.13348, 2025},
 url = {https://arxiv.org/abs/2510.13348},
 arXiv = {arXiv:2510.13348}
}

@article{BB,
  title={On random $2$-adjacent $0/1$-polyhedra.},
  author={Bondarenko, V.~A. and Brodskii, A.~G.},
  journal={Discrete Math. Appl.},
  volume={18},
  number={2},
  year={2008}
}

@misc{BEF25,
  author        = {Babecki, C. and Elling, T. and Ferber, A.},
  title         = {Sharp Threshold for Cliques in Random $0/1$ Polytope Graphs},
  year          = {},
  eprint        = {2507.03212},
  archivePrefix = {arXiv},
  primaryClass  = {math.CO},
  note={arxiv:2507.03212, 2025},
}

@inproceedings{FM92,
  author    = {Feder, T. and Mihail, M.},
  title     = {Balanced Matroids},
  booktitle = {Proceedings of the 24th Annual ACM Symposium on Theory of Computing (STOC)},
  year      = {1992},
  pages     = {26--38},
  publisher = {ACM},
}

@phdthesis{Gillmann,
author = {R. Gillmann},
  title  = {$0/1$-Polytopes: Typical and Extremal Properties},
  school = {Technische Universit{\"a}t Berlin},
  year   = {2006}
}

@article{Harper,
  author  = {Harper, L. H.},
  title   = {Optimal Numberings and Isoperimetric Problems on Graphs},
  journal = {J. Comb. Theory},
  volume  = {1},
  number  = {3},
  pages   = {385--393},
  year    = {1966},
}

@article{HLW,
  author  = {Hoory, S. and Linial, N. and Wigderson, A.},
  title   = {Expander Graphs and Their Applications},
  journal = {Bull. Am. Math. Soc.},
  volume  = {43},
  number  = {4},
  pages   = {439--561},
  year    = {2006},
}

@article{JSV,
  author  = {Jerrum, M. and Sinclair, A. and Vigoda, E.},
  title   = {A Polynomial-Time Approximation Algorithm for the Permanent of a Matrix with Nonnegative Entries},
  journal = {J. ACM},
  volume  = {51},
  number  = {4},
  pages   = {671--697},
  year    = {2004},
}

@incollection{Kaibel04,
  author    = {Kaibel, V.},
  title     = {On the Expansion of Graphs of $0/1$-Polytopes},
  booktitle = {The Sharpest Cut: The Impact of Manfred Padberg and His Work},
  series    = {MPS-SIAM Series on Optimization},
  pages     = {401--415},
  publisher = {SIAM},
  year      = {2004},
}

@inproceedings{KR03,
  author    = {Kaibel, V. and Remshagen, A.},
  title     = {On the Graph-Density of Random $0/1$-Polytopes},
  booktitle = {Approximation, Randomization, and Combinatorial Optimization (APPROX-RANDOM)},
  series    = {Lecture Notes in Computer Science},
  volume    = {2764},
  pages     = {318--328},
  publisher = {Springer},
  year      = {2003},
}

@incollection{Krivelevich,
  author    = {Krivelevich, M.},
  title     = {Expanders---How to Find Them, and What to Find in Them},
  booktitle = {Surveys in Combinatorics 2019},
  series    = {London Mathematical Society Lecture Note Series},
  volume    = {456},
  pages     = {158--205},
  publisher = {Cambridge University Press},
  year      = {2019},
}

@article{LR22,
  author  = {Leroux, B. and Rademacher, L.},
  title   = {Expansion of Random $0/1$ Polytopes},
  journal = {Random Struct. Algorithms},
  volume  = {64},
  number  = {1},
  year    = {2024},
  eprint        = {2207.03627},
  archivePrefix = {arXiv},
  primaryClass  = {math.CO},
}

@inproceedings{MV92,
  author    = {Mihail, M. and Vazirani, U.},
  title     = {On the Expansion of $0/1$ Polytopes},
  booktitle = {Proceedings of the 24th Annual ACM Symposium on Theory of Computing (STOC)},
  pages     = {26--38},
  publisher = {ACM},
  year      = {1992},
}

@inproceedings{Mihail92,
  author    = {Mihail, M.},
  title     = {On the Expansion of Combinatorial Polytopes},
  booktitle = {Mathematical Foundations of Computer Science (MFCS)},
  series    = {Lecture Notes in Computer Science},
  volume    = {629},
  pages     = {37--49},
  publisher = {Springer},
  year      = {1992},
}

@phdthesis{Mihailthesis,
  author = {Mihail, M.},
  title  = {Combinatorial Aspects of Expanders},
  school = {Harvard University},
  year   = {1989},
}

@book{Schrijver,
  author    = {Schrijver, A.},
  title     = {Combinatorial Optimization: Polyhedra and Efficiency},
  publisher = {Springer},
  address   = {Berlin},
  year      = {2003},
}

@article{War16,
  title={On the method of typical bounded differences},
  author={Warnke, L.},
  journal={Comb. Probab. Comput.},
  volume={25},
  number={2},
  pages={269--299},
  year={2016},
  publisher={Cambridge University Press}
}

@misc{GT26,
  title={Random 0/1-polytopes expand rapidly},
  author={Guo, H. and Tomon, I.},
  note={arXiv:2604.09520, 2026},
  year={}
}
\end{document}